\documentclass[10pt]{article}
\usepackage{times,amsfonts,amsmath,amssymb,exscale,mathrsfs}
\usepackage{pst-all}
\usepackage{graphicx}
\usepackage{xcolor} 
\usepackage{epsfig}
\usepackage{shuffle}
\newtheorem{theo}{Theorem}
\newtheorem{prop}[theo]{Proposition}
\newtheorem{lemma}[theo]{Lemma}
\newtheorem{rem}[theo]{Remark}
\newenvironment{proof}{\noindent{\bf Proof:}}{$\Box$\medskip}
\newcommand{\A}{\mathcal{A}}
\newcommand{\Adet}{\mathcal{A_{\rm det}}}
\newcommand{\Asto}{\mathcal{A_{\rm sto}}}
\newcommand{\eA}{\bar{\mathcal{A}}}
\newcommand{\eAdet}{\bar{\mathcal{A}}_{\rm det}}
\newcommand{\eAsto}{\bar{\mathcal{A}}_{\rm sto}}
\newcommand{\B}{\mathcal{B}}
\newcommand{\E}{\mathbb{E}}
\newcommand{\Gq}{\mathcal{G}_{\rm qsh}}
\newcommand{\Gs}{\mathcal{G}_{\rm sh}}
\newcommand{\Hq}{\mathcal{H}_{\rm qsh}}
\newcommand{\Hs}{\mathcal{H}_{\rm sh}}
\newcommand{\N}{\mathbb{N}}
\newcommand{\Nh}{{\mathbb{N}/2}}
\newcommand{\R}{\mathbb{R}}
\newcommand{\sh}{\shuffle}
\newcommand{\W}{\mathcal{W}}

\title{Word combinatorics for stochastic differential equations: splitting integrators
}
\author{
A. Alamo\footnote{Departamento de Matem\'atica Aplicada e IMUVA, Facultad de Ciencias, Universidad de Valladolid,  Spain. Email: alamozapaterouva@gmail.com}
\  and J.M. Sanz-Serna\footnote{Departamento de Matem\'aticas, Universidad Carlos III de Madrid, Avenida de la Universidad 30, E-28911 Legan\'es (Madrid), Spain. Email: jmsanzserna@gmail.com}
}

\date{\today}
\begin{document}
\maketitle

\begin{abstract}
We present an analysis based on word combinatorics of splitting integrators for Ito or Stratonovich systems of stochastic differential equations. In particular we present a technique to write down systematically the expansion of the local error; this makes it possible to easily formulate the conditions that guarantee that a given integrator achieves a prescribed strong or weak order.
This approach bypasses the need to use the Baker-Campbell-Hausdorff (BCH) formula and shows the existence of an order barrier of two for the attainable weak order. The paper also provides a succinct introduction to the combinatorics of words.
\end{abstract}
\bigskip



\section{Introduction}

This paper shows how word combinatorics is a useful tool in the analysis of splitting integrators for Ito or Stratonovich systems of stochastic differential equations. In particular we present a technique  to write down systematically the expansion of the local error; this makes it possible to easily formulate the conditions that guarantee that a given integrator achieves a prescribed strong or weak order. This approach bypasses the need to use the Baker-Campbell-Hausdorff (BCH) formula and shows the existence of an order barrier of two for the attainable weak order. In the case of Stratonovich systems the technique has already appeared in \cite{alfonso}; the corresponding Ito results appear here for the first time. In addition, while the succinct presentation in \cite{alfonso} focuses on the \lq\lq recipe\rq\rq\ to write down the order conditions, the present paper includes  background on the combinatorics of words. In this way  we also provide what we hope is a reader-friendly introduction to that area, which has applications outside numerical mathematics in many mathematical tasks, including averaging of periodically or quasiperiodically forced systems of differential equations, reduction of continuous or discrete dynamical systems to normal form, rough path theory, etc. (references are given below).

The importance of splitting integrators \cite{survey,survey2} has increased continuously in the recent past
due to their flexibility to adapt to the structure of the problem being solved, be it in the context of
multiphysics systems or in the domain of geometric integration  (i.e.\ integration performed under the
requirement that the numerical solution has some of the geometric properties possessed by the true solution)
\cite{gi,hlw,blacalibro}. As it is the case with any other one-step integrator, the analysis of a splitting
algorithm starts with the study of the local error \cite{Bu2016,HaNoWa1993}, i.e.\ the error under the
assumption that the computation at  time level \(t_{n+1}\) starts from information at time \(t_n\) that is
free of errors. Unfortunately, even in the case where the system being integrated consists of (deterministic)
ordinary differential equations, the investigation of the local error may be a daunting task if undertaken in
a naive way. Formal series and combinatorial algebra have been very useful tools as we discuss presently;  see
\cite{china} for a recent survey.

For  Runge-Kutta methods, whose history goes back to 1895, the structure of the local error was only
understood after Butcher's work in the 1960's \cite{butcher63}; this work made it possible to construct
formulas that improve enormously on those known until then. In Butcher's theory, the true and numerical
solutions are expanded in series; each term of the series is the product of a power of the step size, a
numerical coefficient (elementary weight) and a vector-valued function (elementary differential). There is
term in the series associated with each rooted tree. The elementary differentials change with the system being
integrated but are common to all Runge-Kutta formulas and to the true solution. The weights change with the
integrator but are independent of the system being integrated. B-series \cite{HW}, formal series indexed by
rooted trees, were introduced by Hairer and Wanner as a means to systematize Butcher's approach and to extend
it to more general classes of algorithms. B-series are indexed by rooted trees and are combinations of
elementary differentials. A key result in the theory of B-series is the rule to compose two B-series to obtain
a third. B-series possess many applications in numerical analysis, especially in relation to geometric
integration (starting with \cite{canonical}) and modified equations \cite{aust}. (Loosely speaking the
modified equation of a numerical integration is the differential equation exactly satisfied by the numerical
solution.) Recently B-series have also been used outside numerical mathematics, e.g.\ to perform high-order
averaging of periodic or quasiperiodic systems \cite{part1,part2}.

For splitting integrations of deterministic systems, the best-known method  to investigate the local error
\cite{gi} uses the  BCH formula \cite{ssc,hlw}. This may be considered in indirect
approach, in that it does not compare the numerical and true solutions but rather the modified system of the
integrator and the true system being solved. The large combinatorial complexity of the BCH formula is
certainly a limitation of this technique. An alternative methodology, patterned after Butcher's treatment of
the Runge-Kutta case was introduced in \cite{phil} (a summary may be seen in \cite[section III.3]{hlw}). A
third possibility is the use of \emph{word series} expansions
\cite{anderfocm,orlando,part3,juanluis,guirao,words,k,k2}. Word series are patterned after B-series; rather
than combining elementary differentials they combine \emph{word basis functions}. They are indexed by words on
an alphabet rather than by rooted trees. Their scope is narrower than that of B-series; all problems that may
be treated by word series are amenable to analysis via B-series, but the converse is not true. On the other
hand, word series, when applicable, are more compact and simpler to use than B-series; in particular the
composition rule for word series is much simpler than the corresponding rule for B-series. Word series may be
used outside numerical mathematics in tasks such as high-order averaging \cite{orlando,part3,guirao,k,k2},
reduction of dynamical systems to normal form \cite{juanluis}, etc. They are very well suited to investigate
the local error of splitting algorithms \cite{words} (see also the closely related technique in \cite[Section
2.4]{maka}).

Turning now our attention to splitting algorithms for  stochastic differential equations, the most popular technique is again based in the BCH formula, see e.g.~\cite{Leim,Leim2}. In \cite{alfonso} we
suggested a word-series approach in the case where the equations are interpreted in the sense of Stratonovich. This approach bypasses the use of the BCH formula and it is not difficult to implement in practice. Here we extend the material in \cite{alfonso} in several directions that we now discuss briefly.

This paper contains nine sections.
Section 2 recalls the Taylor expansion of the solution of Stratonovich and Ito equations and introduces much
of the notation to be used throughout the paper. In Section 3 we present splitting integrators and their local
errors. We also discuss briefly the pullback operator associated with a mapping; this is a key notion in what
follows, as the local error is investigated here by expanding pullback operators rather than mappings. Section
4 describes the main tool: formal series indexed by words. We employ two kinds of such series: series of
differential operators and series of mappings. The central results, i.e.\ the structure of the strong and weak
local error and the strong and weak order conditions, are given in Section 5. In the Stratonovich case the
order conditions have been already presented in \cite{alfonso}; the Ito case is new, as is the detailed
discussion of the necessity of the order conditions (Lemma~\ref{lem:independence}). Section 6 deals with the
shuffle and quasishuffle products; these play a key role in the combinatorics of words. In our context they
are necessary to identify sets of \emph{independent} order conditions, a point not discussed in \cite{alfonso}, and to prove the composition rule for word series
(Proposition~\ref{prop:composition}). The discussion of the order conditions finishes in
Section 7 with the help of the infinitesimal generator. There we show an order barrier of 2 for the weak order
attainable by splitting integrators in both the Stratonovich and Ito cases.  Sections 8 and 9 present come complements;
they respectively discuss how the relation between the Ito and Stratonovich interpretations may be understood
in terms of word combinatorics and the links between the material in this paper and the theory of Hopf
algebras.

We close the introduction with some important points.
 \begin{itemize}
 \item The word \lq\lq formal\rq\rq\ is often used in some disciplines, such as theoretical physics, as
     somehow synonymous to imprecise or lacking in rigour. In this paper formal series are well defined
     objects that, after truncation, yield meaningful approximations; they are manipulated rigorously
     because all the necessary computations involve \emph{finite} sums.
 \item Our interest is in the  \emph{combinatorial} aspects of the theory. Therefore we shall not concern
     ourselves with the derivation of error bounds or other \emph{analytic} considerations. The interested
     reader is referred to the appendix of \cite{alfonso} (see also \cite{orlando}).
 \item In order not to clatter the exposition, all functions that appear are assumed to be smooth  in the
     whole of the Euclidean space. At some places only a finite number of the terms in some  series make sense if the given vector fields have limited smoothness. In those circumstances one has to replace the series by a finite sum.
\end{itemize}

\section{Stochastic Taylor expansions}
We are concerned with Stratonovich,
\begin{equation}\label{eq:s}
dx = f(x)\,dt+ \sum_{i=1}^n g_i(x) \circ d\B_i,
\end{equation}
or Ito,
\begin{equation}\label{eq:i}
dx = f(x)\,dt+ \sum_{i=1}^n g_i(x)\,  d\B_i,
\end{equation}
systems of differential equations (see e.g.\ \cite{milstein}), where \( f\), \(g_i\), \(i=1,\dots,n\), are
smooth vector fields in \(\R^d\) and \(\B_i\), \(i=1,\dots,n\), are independent scalar Wiener processes. When
applying splitting integrators, \(f\) is often written
 a sum \(\sum_{j=1}^m f_j\); it is then convenient to work hereafter with the formats
\begin{equation}\label{eq:stratonovich}
dx =\sum_{a\in\Adet} f_a(x)\, dt+\sum_{A\in\Asto} f_A(x)\circ d\B_A
\end{equation}
or
\begin{equation}\label{eq:ito}
dx =\sum_{a\in\Adet} f_a(x)\, dt+\sum_{A\in\Asto} f_A(x)\, d\B_A.
\end{equation}
The finite set of indices \(\Adet\) is called the \emph{deterministic alphabet}; its elements are called
\emph{deterministic letters}. The finite set \(\Asto\) is  the \emph{stochastic alphabet} and its elements are
the \emph{stochastic letters}. The set \(\A= \Adet \cup\Asto\) is called the \emph{alphabet} and is assumed to
be nonempty. On the other hand, we include the cases where \(\Adet\) or \(\Asto\) are empty; if
\(\Asto=\emptyset\) then \eqref{eq:stratonovich}--\eqref{eq:ito} is a system of ordinary differential
equations. We use lower case \(a\), \(b\), \dots\ for deterministic letters and upper case \(A\), \(B\),
\dots\ for stochastic letters. The symbols \(k\), \(\ell\), \(m\), \dots\ are used to refer to elements of
\(\A\), i.e.\ to letters, when it is not necessary to specify if they are deterministic or stochastic.

In this section we recall the  expressions of the Taylor expansions of the solutions of
\eqref{eq:stratonovich} or \eqref{eq:ito}  presented in e.g.\ \cite[Chapter 5]{kloeden}. Our treatment is
somewhat different, because we deal with the format \eqref{eq:stratonovich}--\eqref{eq:ito} rather than with
the standard \eqref{eq:s}--\eqref{eq:i}. Specifically, as distinct from \cite{kloeden}, we work here with
deterministic alphabets \(\Adet\) that may have several letters and, in the Ito case, introduce introduce a
letter \(\bar A\) for each \(A\in\Asto\). In the presentation of the Taylor expansion we shall encounter
words, and their differential operators and iterated integrals; these are essential later in the paper.

\subsection{The Stratonovich-Taylor expansion}
With each  letter \(\ell\in\A\)   we associate a first-order differential operator \(D_\ell\). By definition,
\(D_\ell\) is the Lie operator that maps each smooth function  \(\chi:\R^d\rightarrow \R\) into the
function \(D_\ell\chi\) that at the point \(x\in\R^d\) takes the value
\begin{equation}\label{eq:differentialoperator}
D_\ell \chi(x) = \sum_{i=1}^d f_\ell^i(x)\frac{\partial}{\partial x^i}\chi(x) = \chi^\prime(x) f_\ell(x)
\end{equation}
(superscritps denote components of vectors). In \eqref{eq:differentialoperator}, the symbol \(\chi^\prime\) denotes the first (Fr\'{e}chet) derivative of \(\chi\); its value at \(x\in\R^d\) is a linear map defined on \(\R^d\) and
  \(\chi^\prime(x) f_\ell(x)\) is the image by this linear map of the vector \(f_\ell(x)\in\R^d\). Smooth functions \(\chi:\R^d\rightarrow \R\) will often  be referred
to as \emph{observables}. Since the Stratonovich calculus follows the rules of ordinary calculus, if \(x(t)\)
is a solution of \eqref{eq:stratonovich} and \(t_0\geq 0\), \(h\geq 0\),
\begin{eqnarray}\nonumber
\chi(x(t_0+h)) &= &\chi(x(t_0))+\int_{s_1=t_0}^{t_0+h} \chi^\prime(x(s_1))\, dx(s_1) \\
\nonumber
&=& \chi(x(t_0)) + \int_{s_1=t_0}^{t_0+h} \sum_{a\in\Adet} D_a\chi(x(s_1))\, ds_1\\
\nonumber
&& \qquad\qquad\qquad\qquad+\int_{s_1=t_0}^{t_0+h} \sum_{A\in\Asto} D_A\chi(x(s_1))\circ d\B_A(s_1)\\
\label{eq:intchi}
&=& \chi(x(t_0)) + \int_{s_1=t_0}^{t_0+h} \sum_{\ell_1\in\A} D_{\ell_1}\chi(x(s_1))\circ d\B_{\ell_1}(s_1),
\end{eqnarray}
where for deterministic \(\ell_1\) the notation \(\circ d\B_{\ell_1}(s_1)\) means \(ds_1\). In
\eqref{eq:intchi}, as \(h\downarrow 0\), the term \(\chi(x(t_0))\) provides the Taylor approximation of
order 0 to \(\chi(x(t_0+h))\) and the integral gives the corresponding remainder.
 To obtain additional terms of the Taylor expansion of  \(\chi(x(t_0+h))\), we first write formula \eqref{eq:intchi}
with
 \(D_{\ell_1}\chi(x(s_1))\)  in lieu of \(\chi(x(t_0+h))\),
\[
D_{\ell_1}\chi(x(s_1)) = D_{\ell_1}\chi(x(t_0))+\int_{s_2=t_0}^{s_1}
\sum_{\ell_2\in\A} D_{\ell_2}D_{\ell_1}\chi(x(s_2))\circ d\B_{\ell_2}(s_2),
\]
and then substitute in \eqref{eq:intchi} to get
\begin{eqnarray*}
\chi(x(t)) &= &\chi(x(t_0))+ \sum_{\ell_1\in\A} \left(\int_{s_1=t_0}^{t_0+h} \circ d\B_{\ell_1}(s_1)\right)D_{\ell_1}\chi(x(t_0))\\
&&\qquad + \sum_{\ell_1,\ell_2\in\A} \int_{s_1=t_0}^{t_0+h}\circ d\B_{\ell_1}(s_1)
\int_{s_2=t_0}^{s_1}D_{\ell_2}D_{\ell_1}\chi(x(s_2))\circ d\B_{\ell_2}(s_2).
\end{eqnarray*}
By iterating this procedure, we find the series
\begin{equation}\label{eq:taylor}
\chi(x(t_0)) +\sum_{n=1}^\infty\: \sum_{\ell_1,\dots ,\ell_n\in\A}J_{\ell_n\dots \ell_1}(t_0+h;t_0)
D_{\ell_n}\cdots D_{\ell_1}\chi(x(t_0)),
\end{equation}
where \(J_{\ell_n\dots \ell_1}(t_0+h;t_0)\) denotes the \emph{iterated stochastic integral}
\begin{equation}\label{eq:iteratedintegral}
J_{\ell_n\dots \ell_1}(t_0+h;t_0) = \int_{s_1=t_0}^{t_0+h}\circ d\B_{\ell_1}(s_1) \cdots \int_{s_n=t_0}^{s_{n-1}}\circ d\B_{\ell_n}(s_n).
\end{equation}
Iterated integrals obey the following recursion, \(n\geq 2\),
\begin{equation}\label{eq:recurrenceintegrals}
J_{\ell_n\dots \ell_1}(t_0+h;t_0) = \int_{t_0}^{t_0+h}J_{\ell_n\dots \ell_2}(s;t_0)\circ d\B_{\ell_1}(s).
\end{equation}

\begin{rem}\label{rem:integralsbrownian}In the right-hand side of \eqref{eq:taylor} the iterated
integrals are constructed from the Brownian processes \(\B_A\), \(A\in\Asto\), in \eqref{eq:stratonovich} and do not change if the fields \(f_\ell\), \(\ell\in\A\), (or even their dimension \(d\)) change. On the other hand the operators \(D_\ell\) are constructed from the vector fields and do not change with the Brownian processes.
\end{rem}

In the deterministic case, iterated integrals were introduced and investigated extensively by  Kuo Tsai Chen
\cite{chen} in the context of his work on topology.

The notation may be simplified  by introducing the set \(\W\) consisting of all \emph{words} \(\ell_n\ell_{n-1}\dots\ell_1\) constructed
 with the letters of the alphabet \(\A\); \(\W\) includes an empty word \(\emptyset\) with \(n=0\) letters.
 Elements
  \(\ell\in\A\) are seen as  words with a single letter and accordingly \(\A\) becomes a subset of \(\W\).
With each word \(w=\ell_n\dots\ell_1\) with \(n\geq 1\) letters, we associate the \(n\)-th order (linear)
 \emph{differential operator} \(D_w = D_{\ell_n}\cdots D_{\ell_1}\). For the empty word, we define \(D_\emptyset\)
 to be the identity operator \(Id\)
 with \(Id \chi = \chi\) for each  observable and set \(J_\emptyset =1\). (Then \eqref{eq:recurrenceintegrals}
 also holds for \(n=1\)). With this notation the series in \eqref{eq:taylor} simply reads
\begin{equation}\label{eq:taylor2}
\sum_{w\in\W}J_w(t_0+h;t_0)
D_w\chi(x(t_0)).
\end{equation}

We note that for a deterministic letter,
\[
J_a(t_0+h;t_0) = \int_{t_0}^{t_0+h} ds_1 = h,
\]
while in the stochastic case
\[
J_A(t_0+h;t_0) = \int_{t_0}^{t_0+h} \circ d\B_A(s_1) = \B_A(t_0+h)-\B_A(t_0)
\]
is a  Gaussian random variable with standard deviation \(h^{1/2}\). For this reason, we attach to each
deterministic letter \(a\in\Adet\) the \emph{weight} \(\|a\| = 1\) and  each stochastic letter \(A\in\Asto\)
the weight \(\|A\| = 1/2\). We then define the weight \(\|w\|\) of each word
 by adding the weights of its letters. The weight of the empty word is \(0\). The following proposition, whose proof may be seen in \cite{alfonso}, lists some properties of the iterated integrals. It shows in particular that, as \(h\downarrow 0\), \(J_w(t_0+h;t_0)\) may be conceived as
  having size \(\mathcal{O}(h^{\|w\|})\).
\begin{prop}\label{prop:iteratedintegrals}
The iterated Stratonovich integrals \(J_w(t_0+h;t_0)\) have the following properties:
\begin{itemize}
\item The joint distribution of any finite subfamily of the family of random variables \( \{ h^{-\|w\|} J_w(t_0+h;t_0)\}_{w\in\W}\) is independent
of \(t_0\geq 0\) and \( h>0\).
\item \( \E\mid J_w(t_0+h;t_0)\mid^p<\infty\), for each \( w\in\W\), \(t_0\geq 0\), \( h\geq 0 \) and \(p\in
    [0,\infty)\).
\item For each \( w\in\W\) and any finite \( p \geq 1\), the  (\(t_0\)-independent) \(L^p\) norm of the random variable \( J_w(t_0+h;t_0)\) is \(\mathcal{O}(h^{\|w\|})\), as \( h\downarrow 0\).
\item
\( \E\ J_w(t_0+h;t_0) = 0\) whenever \( \|w\|\) is not an integer.
\end{itemize}
\end{prop}

In view of the proposition we rewrite \eqref{eq:taylor2} as:
\begin{equation}\label{eq:taylornu}
\sum_{\nu\in\Nh}\:
\sum_{w\in\W,\,{\|w\|=\nu}}
J_w(t_0+h;t_0)
D_w\chi(x(t_0)),
\end{equation}
where \(\Nh =  \{ 0, 1/2, 1, 3/2,\dots\}\). (For each \(\nu\), the inner sum only contains a finite number of
terms.) In this way, by discarding the terms with \(\nu>\nu_0\) in \eqref{eq:taylornu}, one obtains the Taylor
approximation of order \(\nu_0\) for \(\chi(x(t))\). Of course the series in \eqref{eq:taylornu} in general
does not converge to \(\chi(x(t_0+h))\); it is a \emph{formal series}, whose truncations provide the required
Taylor approximations.

So far it has been assumed that \(\chi\) is scalar-valued. For a vector-valued \(\chi\), the Taylor expansion
is also given by \eqref{eq:taylornu}, with the differential operators \(D_w\)  defined to act
componentwise. The particular choice where \(\chi:\R^d \rightarrow \R^d\) is taken to be the identity function
\(x\mapsto x\), yields the expansion of the solution \(x(t_0+h)\) itself given by
\begin{equation}\label{eq:taylornusol}
\sum_{\nu\in\Nh}\:
\sum_{w\in\W,\,{\|w\|=\nu}}
J_w(t_0+h;t_0)
f_w(x(t_0)),
\end{equation}
where,  \(f_w(x(t_0))\) denotes the result of  applying \(D_w\)
 to the identity function and then evaluating at \(x(t_0)\). Note that the functions
\(f_w\) may be constructed  from the fields \(f_\ell\) in \eqref{eq:stratonovich} with the help of the recursion
\begin{equation}\label{eq:recurrencebasis}
f_{\ell_n\dots \ell_1}(x) = f^\prime_{\ell_{n-1}\dots \ell_1}(x)f_{\ell_n}(x),\qquad n\geq 1,
\end{equation}
where \(f^\prime_{\ell_{n-1}\dots \ell_1}(x)\) stands for the value at \(x\) of the Jacobian matrix of \(
f_{\ell_{n-1}\dots \ell_1}\).
\subsection{The Ito-Taylor expansion}The Taylor expansion of the solution of Ito stochastic differential
system was first derived by Platen and Wagner \cite{pw}.  For \eqref{eq:ito}, formula \eqref{eq:intchi} has to
be replaced by
\begin{eqnarray}
\chi(x(t_0+h))
&=& \chi(x(t_0)) + \int_{s_1=t_0}^{t_0+h} \sum_{a\in\Adet} D_a\chi(x(s_1))\, ds_1\nonumber\\
&& \phantom{\chi(x(t_0))}+ \int_{s_1=t_0}^{t_0+h}\sum_{A\in\Asto} D_A\chi(x(s_1)) d\B_A(s_1)\nonumber\\
&&\phantom{\chi(x(t_0))}+ \int_{s_1=t_0}^{t_0+h} \sum_{A\in\Asto} D_{\bar A}\chi(x(s_1))\, ds_1;\label{eq:intchiito}
\end{eqnarray}
the last term in the right-hand side is the Ito correction, where, for each stochastic letter \(A\), \(D_{\bar
A}\) represents the second-order, linear differential operator
\begin{equation}\label{eq:differentialoperatorbis}
D_{\bar A} \chi(x) = \frac{1}{2} \sum_{i,j=1}^d f_A^i(x)f_A^j(x) \frac{\partial^2}{\partial x^i\partial x^j}\chi(x)
= \frac{1}{2} \chi^{\prime\prime}(x) [f_A(x),f_A(x)].
\end{equation}
The symbol \(\chi^{\prime\prime}\) represents the second (Fr\'{e}chet) derivative of \(\chi\); its value \(\chi^{\prime\prime}(x)\) at a point \(x\in\R^s\) is a bilinear map defined on \(\R^d\times \R^d\) and
\(\chi^{\prime\prime}(x) [f_A(x),f_A(x)]\) is the image by this map of the pair of vectors \([f_A(x),f_A(x)]\).

In order to write \eqref{eq:intchiito} more compactly, we introduce the \emph{extended alphabet}
\(\eA=\eAdet\cup\eAsto\). The extended set \(\eAsto\) of stochastic letters coincides with the old \(\Asto\),
i.e.\ with the set of indices in the second sum in \eqref{eq:ito}; the extended set \(\eAdet\) comprises the
indices \(a\) in the first sum in \eqref{eq:ito} \emph{and,} in addition, a letter \(\bar A\) for each \(A\in
\eAsto =\Asto\). With these notations, \eqref{eq:intchiito} becomes
\[
 \chi(x(t_0+h)) = \chi(x(t_0)) + \int_{s_1=t_0}^{t_0+h} \sum_{\ell_1\in\eA} D_{\ell_1}\chi(x(s_1))\, d\B_{\ell_1}(s_1)
 \]
 (\(d\B_{\ell_1}(s_1) =ds_1\) for \(\ell_1\in\eAdet\)); this is the Ito counterpart of the right-most expression in
 \eqref{eq:intchi}. By
 iterating as in the Stratonovich case, we obtain the series
 \begin{equation}\label{eq:taylorito}
\chi(x(t_0)) +\sum_{n=1}^\infty\: \sum_{\ell_1,\dots ,\ell_n\in\eA}I_{\ell_n\dots \ell_1}(t_0+h;t_0)
D_{\ell_n}\cdots D_{\ell_1}\chi(x(t_0))
\end{equation}
where \(I_{\ell_n\dots \ell_1}(t_0+h;t_0)\) denotes the Ito \emph{iterated stochastic integral}
\[
I_{\ell_n\dots \ell_1}(t_0+h;t_0) = \int_{s_1=t_0}^{t_0+h} d\B_{\ell_1}(s_1) \cdots \int_{s_n=t_0}^{s_{n-1}} d\B_{\ell_n}(s_n).
\]
These iterated integrals satisfy the obvious analogue of the recursion \eqref{eq:recurrenceintegrals}. Again the iterated integrals do not change if the vector fields are changed and the operators \(D_\ell\) do not change if the Brownian processes are changed.

We now consider the set of words \(\overline{\W}\) constructed with the letters of the extended  alphabet \(\eA\),
and
write \(D_w = D_{\ell_n}\cdots D_{\ell_1}\) for \(w= \ell_n\dots \ell_1\in\overline \W\), \(n>0\),
 (recall that \(D_\ell\) is a second order operator if \(\ell\) is of the form \(\bar A\),
\(A\in\Asto\)), \(D_\emptyset =Id\), \(I_\emptyset=1\). Then \eqref{eq:taylorito} has the compact expression
\begin{equation}\label{eq:taylor2ito}
\sum_{w\in\overline\W}I_w(t_0+h;t_0)
D_w\chi(x(t_0)).
\end{equation}

If  letters in \(\eAdet\) are again declared to have weight \(1\) and letters in \(\eAsto\) to have weight
\(1/2\), we have the following result, whose proof is similar to that of
Proposition~\ref{prop:iteratedintegrals}:

\begin{prop}\label{prop:iteratedintegralsito}
The Ito iterated integrals \(I_w(t_0+h;t_0)\) possess the properties of the Stratonovich iterated  integrals
listed in Proposition \ref{prop:iteratedintegrals}
\end{prop}

 The series \eqref{eq:taylor2ito} is rewritten as
\begin{equation}\label{eq:taylornuito}
\sum_{\nu\in\Nh}\:
\sum_{w\in\overline\W,\,{\|w\|=\nu}}
I_w(t_0+h;t_0)
D_w\chi(x(t_0)),
\end{equation}
and for the solution itself we have the Taylor series
\[
\sum_{\nu\in\Nh, \nu}\:
\sum_{w\in\overline\W,\,{\|w\|=\nu}}
I_w(t_0+h;t_0)
f_w(x(t_0)),
\]
where \(f_w(x(t_0))\), \(w=\ell_n\dots\ell_1\) denotes the result of successively applying \(D_{\ell_1}\),
\dots, \(D_{\ell_n}\) to the identity function and then evaluating at \(x(t_0)\). The \(f_w\) satisfy
\eqref{eq:recurrencebasis} if \(\ell_n\in \Adet\cup \Asto\) and
\begin{equation}\label{eq:recursionbasistwo}
f_{\ell_n\dots \ell_1}(x) = \frac{1}{2} \Big(f^{\prime\prime}_{\ell_{n-1}\dots \ell_1}(x)\Big)
[ f_A(x),f_A(x)],\qquad n\geq 1,
\end{equation}
for \(\ell_n=\bar A\) with \(A\in\Asto\). Since the second derivatives of the identity function vanish we have the following result.

\begin{prop}If the last (i.e.\ right-most) letter of \(w\in\overline{\W}\) is of the form \(\bar A\) with \(A\in\Asto\), then
\(f_w\) vanish identically.
\end{prop}

Therefore, after suppressing the \(f_w\) that vanish identically, the Taylor expansion may be written:
\begin{equation}\label{eq:taylornusolito}
\sum_{\nu\in\Nh, \nu}\:
\sum_{w\in\overline\W_0,\,{\|w\|=\nu}}
I_w(t_0+h;t_0)
f_w(x(t_0)),
\end{equation}
where \(\overline{\W}_0\) is the subset of \(\overline{\W}\) consisting of words whose last letter is not one of the
\(\bar A\), \(A\in\Asto\).
\section{Analyzing splitting integrators: preliminaries}

\subsection{Splitting integrators}

In order  to avoid notational complications, let us momentarily consider only the simple instance of \eqref{eq:stratonovich} given by
\begin{equation}\label{eq:stratonovichparticular}
dx =f_a(x)\, dt+ f_A(x)\circ d\B_A.
\end{equation}
\emph{Splitting integrators} may be applied to integrate this system if one may solve in closed form the
\emph{split systems}
\begin{equation}\label{eq:splita}
dx =f_a(x)\, dt
\end{equation}
and
\begin{equation}\label{eq:splitA}
dx = f_A(x)\circ d\B_A.
\end{equation}
In the simplest splitting integrator, the Lie-Trotter algorithm, the numerical solution is advanced from its
value \(x_n\) at a time level \(t_n\) to the value \(x_{n+1}\) at the next time level \(t_{n+1}\) by first
integrating \eqref{eq:splita} in the interval \([t_n,t_{n+1}]\) with initial condition \(x_n\) to get a value
\(\tilde {x}_n\) and then using \(\tilde {x}_n\) as initial condition to integrate \eqref{eq:splitA} in the
interval \([t_n,t_{n+1}]\) to obtain \(x_{n+1}\). In this way, the simultaneous contributions of \(f_a\) and
\(f_A\) in \eqref{eq:stratonovichparticular} are replaced by successive contributions. The procedure is best
described by introducing, for \(0\leq s\leq t\), the solution maps \(\varphi_{t,s}^{(a)}\),
\(\varphi_{t,s}^{(A)}\) of \eqref{eq:splita} and \eqref{eq:splitA}; by definition, \(\varphi_{t,s}^{(a)}\)
(respectively \(\varphi_{t,s}^{(A)}\)) maps \(x\in\R^d\) into the value at time \(t\) of the solution of
\eqref{eq:splita} (respectively \eqref{eq:splitA}) with initial value \(x\) at time \(s\). Note that, for the
autonomous deterministic system \eqref{eq:splita}, \(\varphi_{t,s}^{(a)}\) depends  on \(t\) and \(s\) only
through the combination (elapsed time) \(t-s\), but that is not the case for \(\varphi_{t,s}^{(A)}\). In
addition \(\varphi_{t,s}^{(a)}\) makes sense for \(t<s\), but \(\varphi_{t,s}^{(A)}\) does not, because
stochastic differential equations cannot be evolved backward in time. With this notation in place, one step of
the Lie-Trotter algorithm described above is given by \(x_{n+1} =\psi_{t_{n+1},t_n}(x_n)\), where
\begin{equation}\label{eq:lietrotter}
\psi_{t_{n+1},t_n}=\varphi_{t_{n+1},t_n}^{(A)}\circ \varphi_{t_{n+1},t_n}^{(a)}.
\end{equation}
Of course one may also consider the alternative algorithms given by \(\varphi_{t_{n+1},t_n}^{(a)}\circ \varphi_{t_{n+1},t_n}^{(A)}\) or  the well-known symmetric compositions
\[
\varphi_{t_{n+1},t_{n+1/2}}^{(a)}\circ\varphi_{t_{n+1},t_n}^{(A)}\circ \varphi_{t_{n+1/2},t_n}^{(a)}
\]
or
\[
\varphi_{t_{n+1},t_{n+1/2}}^{(A)}\circ\varphi_{t_{n+1},t_n}^{(a)}\circ \varphi_{t_{n+1/2},t_n}^{(A)}
\]
associated with Strang's name (\(t_{n+1/2}\) is the midpoint of \([t_n,t_{n+1}]\)). More involved splitting
algorithms are obtained by composing four or more solution maps of the split systems.

Leaving now the particular instance \eqref{eq:stratonovichparticular}, for a problem of the general form
\eqref{eq:stratonovich}   the splitting-integrator mapping \(x_{n+1}=\psi_{t_{n+1},t_n}(x_n)\) is a
composition of solution operators
\begin{equation}\label{eq:mapscomposed}
\varphi^{(i)}_{t_n+d_i(t_{n+1}-t_n),t_n+c_i(t_{n+1}-t_n)},\qquad i =1,\dots,m.
\end{equation}
Here \(c_i\) and \(d_i\) are constants and the superindex \(i\) refers to a system of differential equations
obtained by taking into account a subset \(\mathcal{S}_i\), \(i=1,\dots,m\) of the fields \(f_\ell\)  in
\eqref{eq:stratonovich}; it has to be supposed that these systems are solvable in closed form. For our
purposes here, there is complete freedom when choosing the different \(\mathcal{S}_i\); it is possible to have
\(\mathcal{S}_i=\mathcal{S}_j\) for \(i\neq j\) (as in Strang's method where \(\mathcal{S}_1=\mathcal{S}_3\))
or to let given vector field \(f_\ell\) appear in \(\mathcal{S}_i\) and \(\mathcal{S}_j\) with
\(\mathcal{S}_i\neq \mathcal{S}_j\). It is important to note that it is necessary to assume throughout that
\[
c_i < d_i
\]
except in the case where \(\mathcal{S}_i\) is a deterministic system; stochastic differential equations cannot
be evolved backward in time.

The Ito case can be dealt with in the same way; the only difference is that the solution operators of the
 systems \(\mathcal{S}_i\) have to be based on the Ito interpretation.

\subsection{The local error}
An essential part of the  analysis of any one-step  integrator \(x_{n+1}=\psi_{t_{n+1},t_n}(x_n)\) is the
study of the corresponding \emph{local error}
 (or truncation error). By definition, if \(\varphi_{t,s}\) denotes the solution operator of the system
 \eqref{eq:stratonovich} or
 \eqref{eq:ito}   being integrated, the local error is the difference
 \begin{equation}\label{eq:localerror}\psi_{t_{n+1},t_n}(x_n)-\varphi_{t_{n+1},t_n}(x_n).\end{equation}
  In what follows we just consider the case \(\psi_{t_{1},t_0}(x_0)-\varphi_{t_{1},t_0}(x_0)\);
  the case with general \(n\) differs from this only in notation. Furthermore we write \(t_1=t_0+h\),
  where \(h>0\) is the step-length. Our aim is to understand the behaviour of \eqref{eq:localerror} as \(h\downarrow 0\) and this is achieved by Taylor expansion. In the particular
case of the Lie-Trotter integrator \eqref{eq:lietrotter} for the simple system
\eqref{eq:stratonovichparticular}, we have therefore to Taylor expand
 \begin{equation}\label{eq:aux0}
 x_1=\varphi_{t_{1},t_0}^{(A)}\Big (\varphi_{t_{1},t_0}^{(a)}(x_0)\Big)
 \end{equation}
 and compare the result with the expansion of  \(\varphi_{t_{1},t_0}(x_0)\)
 found in the preceding section. Note that if we write
 \begin{equation}\label{eq:aux1}\tilde{x}_0 = \varphi_{t_{1},t_0}^{(a)}(x_0),
 \end{equation}
 so that
 \begin{equation}\label{eq:aux2}
  x_1 =\varphi_{t_{1},t_0}^{(A)}(\tilde{x}_0)
  \end{equation}
 the expansions \(\sum_i
 c_i F_i(x_0)\)  of \eqref{eq:aux1} at \(x_0\) and \(\sum_j
 d_j G_j(\tilde{x}_0)\) of \eqref{eq:aux2}
 at \( \tilde{x}_0\) are both known;
 they are particular instances of \eqref{eq:taylornusol} corresponding to alphabets with the single letter
 \(a\) or \(A\) respectively. Then expansion for \eqref{eq:aux0} may be obtained by substituting to get
 \[
 \sum_j
 d_j G_j\big(\sum_i
 c_i F_i(x_0)\big),
 \]
 Taylor expanding  each \(G_j\big( \sum_i
 c_i F_i(x_0)\big)\) and gathering terms of equal weight.
 For more complicated splitting integrators there are \(m\) mappings being composed and implementing the naive
 substitution we have described may be a daunting task.
We are thus led to the following:
\medskip

 \emph{Problem P: Find efficiently the
 expansion
 of a composition of mappings \(\varphi^{(m)}\circ \cdots\circ\varphi^{(1)}\), when
 \(\varphi^{(i)}\), \(i=1,\dots, m\),  have known expansions  of the form \eqref{eq:taylornusol} (or \eqref{eq:taylornusolito}
 for the Ito case).
 }
 \medskip

  The solution to this problem presented in the next section is based on expanding pullback operators
  (see e.g.\ \cite{k2}) rather
  than mappings.

 \subsection{Pullbacks}

 Associated with any mapping \(\varphi:\R^d\rightarrow \R^d\), there is a pullback operator \(\Phi\). By definition,
  \(\Phi\)  maps each observable \(\chi\) into the observable \(\Phi \chi\) whose value at \(x\in\R^d\) is
 (\(\Phi\chi)(x) = \chi(y)\) with \(y=\varphi(x)\) (\(\varphi\) pushes the point \(x\) forward to \(y\), while \(\Phi\)
  takes  values of the observable \lq\lq back\rq\rq\ from \(y\) to \(x\)). The pullback operator
  corresponding to a composition \(\varphi^{(2)}\circ\varphi^{(1)}\) is the composition of operators
  \(\Phi^{(1)}\Phi^{(2)}\) (note the reversed order) because
  \[
\Big(\Phi^{(1)}(\Phi^{(2)}\chi)\Big) (x) = (\Phi^{(2)} \chi)(\varphi^{(1)}(x))
=\chi\Big(\varphi^{(2)}(\varphi^{(1)}(x))\Big).
  \]

  A map and its pullback operator contain the same information: when the operator \(\Phi\) is known,
  one may retrieve the underlying map \(\varphi\) by applying \(\Phi\)
  to the identity  \(x\mapsto x\) in \(\R^s\).
Recovering \(\Phi\) from \(\varphi\) is  similar to what was done \emph{for formal series} rather than for
maps to obtain \eqref{eq:taylornusol} from
  \eqref{eq:taylornu} (or \eqref{eq:taylornusolito} from \eqref{eq:taylornuito} in the Ito case). Taking this point
  further, from \eqref{eq:taylornu} we may consider that the series
  \begin{equation}\label{eq:expansionpullback}
  \sum_{\nu\in\Nh}\:
\sum_{w\in\W,\,{\|w\|=\nu}}
J_w(t_0+h;t_0)
D_w,
  \end{equation}
 provides the Taylor
  expansion of the pullback operator of the solution operator \(\varphi_{t_0+h,t_0}\) of \eqref{eq:stratonovich}.
  For the Ito case \eqref{eq:expansionpullback} is
   replaced by
  \begin{equation}\label{eq:expansionpullbackito}
  \sum_{\nu\in\Nh}\:
\sum_{w\in\overline\W,\,{\|w\|=\nu}}
I_w(t_0+h;t_0)
D_w.
  \end{equation}

  In this way  the problem posed above may be reformulated
  as:\medskip

 \emph{Problem P': Find efficiently the expansion of a composition \(\Phi^{(1)}\cdots\Phi^{(m)}\) of pullback operators,
 when the operators
 \(\Phi^{(i)}\), \(i=1,\dots, m\),  have known expansions of the form  \eqref{eq:expansionpullback} (or
 \eqref{eq:expansionpullbackito}
 for the Ito case).
 }

 The idea of using pullback (differential) operators to analyze local errors is old. Merson \cite{merson} used
 it in 1957 to study Runge-Kutta formulas; however the subsequent treatment in Butcher \cite{butcher63} did
 away with differential operators and worked only with elementary differentials (mappings). In the stochastic
 case it is convenient to work both with differential operators and mappings, as it will become clear below.

\section{Series}
We now  solve the problem P' (and by implication P) with the help of some simple algebraic/combinatorial
tools.

\subsection{Series for Stratonovich problems}
\subsubsection{Series}

  Words in \(\W\) are multiplied by \emph{concatenation}, i.e.\ if \(v = k_1\dots k_m\), \(w = \ell_1\dots\ell_n\) are words
   with \(m\) and \(n\) letters respectively, their product is the word with \(m+n\) letters
   \(vw = k_1\dots k_m\ell_1\dots\ell_n\). In particular \(\emptyset\emptyset = \emptyset\), \(\emptyset w =
   w\emptyset = w\). Concatenation is associative but it is not commutative.

   The vector space \(\R\langle \A\rangle\) consists, by definition, of all linear combinations of words
\( \sum_{w\in\W} c_w w\) (only a finite number of coefficients \(c_w\in\R\) are nonzero). The multiplication of words by concatenation is extended in an obvious way to elements of \(\R\langle \A\rangle\),
\begin{equation}\label{eq:multiplication}
\sum_{v\in\W} c_v v \sum_{w\in\W} d_{w} w = \sum_{v,w\in\W} c_vd_w uw,
\end{equation}
and then \(\R\langle \A\rangle\) becomes a noncommutative, associative algebra.

In addition, we need to consider the larger noncommutative algebra \(\R\langle\langle \A\rangle\rangle\) of
 formal series. These are formal expressions \( \sum_{w\in\W} c_w w\) where it is not any longer assumed
that only finitely many coefficients \(c_w\) are \(\neq 0\). If \(S\in\R\langle\langle \A\rangle\rangle\), we
denote the corresponding coefficients by \(S_w\), i.e.\ \(S = \sum_{w\in W}S_ww\).
 Formal series are combined linearly in an obvious way and  are multiplied as in
\eqref{eq:multiplication}, where we note that the right-hand side is well defined, even if infinitely many
\(c_v\) and \(d_w\) do not vanish, because the number of ways in which a given \(u\in\W\) may be  written as a
concatenation \(u=vw\) is finite. More precisely, if we denote by \(\R^\W\) the set of all sequences of
coefficients \(\{c_w\}_ {w\in\W}\)  indexed by words,  then the product in \eqref{eq:multiplication} is the
 series \(\sum_{u\in\W} e_u u\in\R\langle\langle\A\rangle\rangle\) with coefficients \(\{e_u\}_{u\in\W}\)
such that \(e_\emptyset = c_\emptyset d_\emptyset\) and , for each nonempty word \(u = \ell_1\dots\ell_n\),
\begin{equation}\label{eq:convolution}
e_{\ell_1\dots\ell_n} = c_\emptyset d_{\ell_1\dots\ell_n}
 + \sum_{m=1}^{n-1} c_{\ell_1\dots\ell_m}d_{\ell_{m+1}\dots \ell_n} + c_{\ell_1\dots\ell_n}d_\emptyset.
\end{equation}
 The right-hand
side of this formula contains all the ways of writing \(u = \ell_1\dots\ell_n\) as a concatenation of two
(possibly empty) words. Thus \eqref{eq:convolution} defines a (noncommutative, associative) product in the set
\(\R^\W\) of sequences of coefficients,
 the so-called \emph{convolution} product, in such a way that
the product of  series \(S\in\R\langle\langle \A\rangle\rangle\)  corresponds to the convolution product
of the sequences of coefficients \(\{S_w\}_ {w\in\W}\in\R^\W\).

A general well-known reference to the combinatorics of words is \cite{reut}.

\subsubsection{Series of differential operators}
  Given the vector fields \(f_\ell\) in \eqref{eq:stratonovich}, the concatenation product
 of words obviously corresponds to the composition of the associated differential operators: \(D_{vw} = D_vD_w\)
  for any \(v,w\in\W\) (by definition, \((D_vD_w)\chi= D_v(D_w\chi)\) for each observable \(\chi\)).

With the series \( S\in \R\langle\langle \A\rangle\rangle\) we associate the formal series of differential
operators
 \(D_S =\sum_{w\in\W} S_w D_w\). It follows that \( S, S^\prime\in
\R\langle\langle \A\rangle\rangle\), the product (composition) \(D_SD_{S^\prime}\) is the series
\(D_{SS^\prime}\) associated with \(SS^\prime\), whose coefficients, as we know, are given by the convolution
 product of the coefficients of
\(S\) and \(S^\prime\). Series of differential operators are a common tool in control theory, see e.g.
\cite{f}.

With the terminology we have introduced, for each fixed \(t\), \(t_0\) and for each event in the underlying
probability space, the expansion in \eqref{eq:expansionpullback} coincides with \(D_S \) when the coefficients
are  \(S_w =J_w(t_0+h;t_0)\), \(w\in\W\). Since we have just described how to multiply series \(D_S\), we have
solved the problem P' posed in the previous section.

We illustrate the technique by means of a simple example. We integrate the system
\[dx = f_a(x)dt+f_b(x)dt+f_A(x)\circ d\B_A,
\]
with the help of the split systems
\[
(1)\:dx = f_a(x)dt+f_A(x)\circ d\B_A,\qquad (2)\:dx = f_b(x)dt.
\]
We use the Lie-Trotter formula \(\varphi^{(2)}\circ\varphi^{(1)}\). According to \eqref{eq:expansionpullback}
(when the alphabet is chosen to be \(\{a,A\}\)), the expansion of \(\Phi^{(1)}\) is
\[
Id+J_AD_A+J_aD_a+J_{AA}D_{AA}+ J_{aA}D_{aA}+J_{Aa}D_{Aa}+J_{AAA}D_{AAA}+\mathcal{O}(2),
\]
where \(\mathcal{O}(2)\) denotes the terms in the series with weight \(\geq 2\), and \(J_A\), \(J_a\), \dots\ stand
for \(J_A(t_0+h;t_0)\), \(J_a(t_0+h;t_0)\), \dots\ Similarly, the expansion of \(\Phi^{(2)}\) is
\[
Id+J_bD_b+\mathcal{O}(2).
\]
Multiplying out, we obtain the expansion for the product \(\Phi^{(1)}\Phi^{(2)}\):
\begin{eqnarray*}
&&Id+J_AD_A+J_aD_a+J_bD_b+J_{AA}D_{AA}\\
&&\qquad+J_{aA}D_{aA}+J_{Aa}D_{Aa}+J_{A}J_bD_{Ab}+J_{AAA}D_{AAA}+\mathcal{O}(2).
\end{eqnarray*}
For the solution of the system being integrated, \eqref{eq:expansionpullback} (when the alphabet is \(\{a,b,A\}\)) yields
\begin{eqnarray*}
&&Id+J_AD_A+J_aD_a+J_bD_b+J_{AA}D_{AA}\\
&&\qquad+J_{aA}D_{aA}+J_{bA}D_{bA}+J_{Aa}D_{Aa}+J_{Ab}D_{Ab}+J_{AAA}D_{AAA}+\mathcal{O}(2),
\end{eqnarray*}
and subtracting we find that the pullback operator associated with the local error has the  expansion:
\begin{equation}\label{eq:localerrorexample}
(J_{A}J_b-J_{Ab})D_{Ab}-J_{bA}D_{bA}+\mathcal{O}(2).
\end{equation}
\subsubsection{Word series}

Given  the vector fields \(f_\ell\) in \eqref{eq:stratonovich}, with each  series \( S\in \R\langle\langle \A\rangle\rangle\) we associate the corresponding \emph{word
series} \(\W_S(x_0)\); this is obtained by applying \(D_S\) to the identity map \(x\in\R^d\mapsto x\):
\[
\W_S(x_0)=\sum_{w\in\W}S_w f_w(x_0).
\]
The functions \(f_w:\R^d\rightarrow \R^d\), \(w\in\W\),  we already encountered in \eqref{eq:taylornusol}, are
called \emph{word basis functions}. Recall that they may be found recursively via \eqref{eq:recurrencebasis}
from the \(f_\ell\) that appear in the system \eqref{eq:stratonovich}. Word series, introduced and studied in
\cite{anderfocm,part2,orlando,alfonso,words,k,k2}, may be seen as equivalent to series of differential
operators; the theory of words series is patterned after the theory of B-series \cite{HW} familiar to many
numerical analysts.

With the terminology above, for fixed \(t_0\) and \(h\) and each event in the underlying probability space,
the expansion \eqref{eq:taylornusol} is simply the word series with coefficients given by the iterated
integrals \(J_w\). In what follows we shall denote by \(J\) the series \(J=\sum_w
J_w(t_0+h;t_0)w\in\R\langle\langle \A\rangle\rangle\), so that \(D_J\) and \(\W_J\) are the corresponding
series of operators and word series. Formal series of words whose coefficients are iterated integrals are
often called \emph{Chen series}; they play a role in several mathematical developments, including the theory
of rough paths (see e.g.\ \cite{b}).

In the example we are discussisng, from \eqref{eq:localerrorexample} we obtain that the local error has the
expansion
\begin{equation}\label{eq:localerrorexamplebis}
(J_{A}J_b-J_{Ab})f_{Ab}-J_{bA}f_{bA}+\mathcal{O}(2),
\end{equation}
with \(f_{Ab} = f^\prime_bf_A\), \(f_{bA} = f^\prime_Af_b\).

\subsection{Series for Ito problems}
The preceding material is easily adapted to the Ito system \eqref{eq:ito}. The required changes are few. One
considers formal series \(S \in \R\langle\langle \eA\rangle\rangle\) (words are now based on the extended
alphabet) and to each \(S= \sum_{w\in\overline \W}S_ww\) associates a series of differential operators \(D_S =
\sum_{w\in\overline \W}S_wD_w\). The expansion \eqref{eq:expansionpullbackito} of the pullback of the solution
operator  is \(D_S\) when the coefficients of the series are chosen to be the Ito iterated  integrals. We
write this series as \(D_I\) and set \(I=\sum_{w\in\overline W} I_w(t_0+h;t_0)w\in \R\langle\langle
\eA\rangle\rangle\) for the corresponding \emph{Chen series}.

Here is an example. For the Ito  system corresponding to the alphabet \(\{a,b,A\}\), split as (1) \(\{a,A\}\),
(2) \(\{b\}\), the expansion of
of \(\Phi^{(1)}\)
\begin{eqnarray*}
&&Id+I_AD_A+I_aD_a+I_{\bar{A}}D_{\bar{A}}+I_{AA}D_{AA}\\
&&\qquad
+ I_{aA}D_{aA}+I_{\bar A A}D_{\bar A A}+I_{Aa}D_{Aa}+
I_{A\bar A}D_{A\bar A}+I_{AAA}D_{AAA}+\mathcal{O}(2),
\end{eqnarray*}
the expansion of \(\Phi^{(2)}\) is
\[
Id+I_bD_b+\mathcal{O}(2),
\]
and, multiplying out, the expansion
\(\Phi^{(1)}\Phi^{(2)}\)
is found to be
\begin{eqnarray*}
&&Id+I_AD_A+I_aD_a+I_{b}D_{b}+I_{\bar A}D_{\bar A}+I_{AA}D_{AA}\\
&&\qquad
+ I_{aA}D_{aA}+I_{\bar A A}D_{\bar A A}+I_{Aa}D_{Aa}+I_{A}I_b D_{Ab}
+I_{A\bar A}D_{A\bar A}+I_{AAA}D_{AAA}\\
&&\qquad
+\mathcal{O}(2).
\end{eqnarray*}
For the solution of the system being integrated we have
\begin{eqnarray*}
&&Id+I_AD_A+I_aD_a+I_{b}D_{b}+I_{\bar A}D_{\bar A}+I_{AA}D_{AA}\\
&&\qquad
+ I_{aA}D_{aA}+I_{bA}D_{bA}+I_{\bar A A}D_{\bar A A}+I_{Aa}D_{Aa}+I_{Ab} D_{Ab}
+I_{A\bar A}D_{A\bar A}\\
&&\qquad
+I_{AAA}D_{AAA}+\mathcal{O}(2),
\end{eqnarray*}
and, for the pullback of the truncation error,
\begin{equation}\label{eq:localerrorexampleito}
(I_{A}I_b-I_{Ab})D_{Ab}-I_{bA}D_{bA}+\mathcal{O}(2),
\end{equation}
while for the truncation error itself we have the word series expansion:
\begin{equation}\label{eq:localerrorexamplebisito}
(I_{A}I_b-I_{Ab})f_{Ab}-I_{bA}f_{bA}+\mathcal{O}(2),
\end{equation}
with \(f_{Ab} = f^\prime_bf_A\), \(f_{bA} =  f^\prime_Af_b\).
\section{The expansion of the local error. Error equations}
\label{sec:ordercond}
In this section we present the Taylor expansion of the local error  along with the  conditions that have to be imposed to achieve a target strong or weak order of consistency.

\subsection{Expanding the local error}

By applying the technique  in the previous section, the Taylor expansion of the mapping \(\psi_{t_0+h;t_0}\) that describes a splitting integrator for the Stratonovich system \eqref{eq:stratonovich} is  found as a word series
\[
\W_{\widetilde J}(x_0) =\sum_{w\in\W} \widetilde{J}_w(t_0+h;t_0) f_w(x_0).
\]
Here  \(\widetilde J_w(t_0;t_0+h)\) is either zero or a product of Stratonovich  iterated integrals
corresponding to words whose concatenation is \(w\) (see \eqref{eq:localerrorexamplebis} for an example).
Thus, in each product, the iterated integrals being multiplied correspond to words whose weights add up to
\(\|w\|\). In particular
\(\widetilde J_\emptyset(t_0;t_0+h)=1\). For the corresponding pullback we have the expansion
\[
D_{\widetilde J} = \sum_{w\in\W} \widetilde{J}_w(t_0+h;t_0) D_w
\]
(see \eqref{eq:localerrorexample}).

Similarly, in the Ito case, \(\psi_{t_0+h;t_0}\) has a word series expansion
\[
\W_{\widetilde{I}}(x_0) =\sum_{w\in\overline{\W}_0} \widetilde{I}_w(t_0+h;t_0) f_w(x_0)
\]
(see \eqref{eq:localerrorexamplebisito}) and for the associated pullback the expansion is
\[
D_{\widetilde I} =\sum_{w\in\overline\W} \widetilde{I}_w(t_0+h;t_0) D_w
\]
(see \eqref{eq:localerrorexampleito}).

The proof of the following technical result may be found in \cite{alfonso} for the Stratonovich case; the Ito
case is proved similarly.

\begin{prop}\label{prop:iteratedintegralsnum}The coefficients \(\widetilde{J}_w(t_0+h;t_0)\), \(w\in\W\),
possess the properties of the exact coefficients \({J}_w(t_0+h;t_0)\) listed in
Proposition~\ref{prop:iteratedintegrals}. The coefficients \(\widetilde{I}_w(t_0+h;t_0)\), \(w\in\overline\W\),
possess the properties of the exact coefficients \(I_w(t_0+h;t_0)\) listed in Proposition~\ref{prop:iteratedintegralsito}.
\end{prop}

By subtracting the expansions of the integrator and the true solution, we immediately obtain the next result. The bound
for \(\|\delta_w(t_0;h)\|_p\) follows from the third item in Proposition~\ref{prop:iteratedintegrals} and the corresponding result for \(\widetilde{J}_w(t_0+h;t_0)\) in Proposition~\ref{prop:iteratedintegralsnum}. Note that the halfinteger
values of \(\nu\) drop from \eqref{eq:mainweak} in view of the last item in Proposition~\ref{prop:iteratedintegrals} and the corresponding result for \(\widetilde{J}_w(t_0+h;t_0)\).
\begin{theo}\label{th:main}
For a splitting integrator for the Stratonovich system \eqref{eq:stratonovich}, the local error \(\psi_{t_0+h;t_0}(x_0)\) has a word series expansion
\begin{equation}\label{eq:mainstrong}
  W_{\delta(t_0;h)}(x_0) = \sum_{\nu\in\Nh,\nu\neq 0}\:\sum_{w\in\W,\|w\|=\nu}\delta_w(t_0;h)f_w(x_0)
\end{equation}
with coefficients
\[
\delta_w(t_0;h)=\widetilde{J}_w(t_0+h;t_0)-J_w(t_0+h;t_0),\qquad w\in\W.
\]
For each nonempty \(w\in\W\) and any \(L^p\) norm \(1\leq p < \infty\), uniformly in \(t_0\),
\[
\|\delta_w(t_0;h)\|_p=\mathcal{O}(h^{\|w\|}),\qquad h\downarrow 0.
\]

In addition, for each observable \(\chi\), conditional on \(x_0\), the error in expectation
\[\E\big(\psi_{t_0+h;t_0}(x_0)\big) - \E\big(\phi_{t_0+h;t_0}(x_0)\big)\]
 has the expansion
\begin{equation}\label{eq:mainweak}
\sum_{\nu\in\N,\nu\neq 0}\:\sum_{{w\in\W,\|w\|=\nu}}\E\big(\delta_w(t_0;h)\big) D_w\chi(x_0).
\end{equation}
\end{theo}

Of course to obtain good strong approximations, integrators with small error coefficients \(\delta_w(t_0;h)\)
are to be preferred, all other things being equal, to integrators with large error coefficients. A similar
comment applies to weak approximations. The reference \cite{alfonso} presents a comparison between two
splitting integrators of the Langevin dynamics introduced by Leimkuhler and Matthews \cite{Leim,Leim}. While
both schemes are closely related, it is found in \cite{Leim,Leim} that, in practice, one clearly outperforms
the other; this is explained in \cite{alfonso} by analysing the corresponding error coefficients.

In the Ito case we have the following result:
\begin{theo}
For a splitting integrator for the Ito system of differential equations \eqref{eq:ito},  the coefficients
\[
\eta_w(t_0;h)=\widetilde{I}_w(t_0+h;t_0)-I_w(t_0+h;t_0),
\]
satisfy, for each nonempty \(w\in\overline\W\) and any \(L^p\) norm \(1\leq p < \infty\), uniformly in \(t_0\),
\[
\|\eta_w(t_0;h)\|_p=\mathcal{O}(h^{\|w\|}),\qquad h\downarrow 0.
\]

The local error \(\psi_{t_0+h;t_0}(x_0)\) has a word series expansion
\begin{equation}\label{eq:mainstrongito}
  W_{\eta(t_0;h)}(x_0) = \sum_{\nu\in\Nh,\nu\neq 0}\:\sum_{w\in\overline{\W}_0,\|w\|=\nu}\eta_w(t_0;h)f_w(x_0)
\end{equation}

In addition, for each observable \(\chi\), conditional on \(x_0\),  the error in expectation
\[\E\big(\psi_{t_0+h;t_0}(x_0)\big) - \E\big(\phi_{t_0+h;t_0}(x_0)\big)\]
 has the expansion
\begin{equation}\label{eq:mainweakito}
\sum_{\nu\in\N,\nu\neq 0}\:\sum_{{w\in\overline{\W},\|w\|=\nu}}\E\big(\eta_w(t_0;h)\big)\, D_w\chi(x_0).
\end{equation}
\end{theo}

\subsection{Stratonovich order conditions}

If \(\mu\in\Nh\), \(\mu>0 \),  we shall say that the integrator has \emph{strong order} \(\geq \mu\)
if the series  \eqref{eq:mainstrong}
only comprises  terms of weight \(\geq \mu+1/2\), i.e.\ of size
 \(\mathcal{O}(h^{\mu+1/2})\).
From Theorem~\ref{th:main} it is clear that for \(\mu\in\Nh\), \(\mu>0 \), the \emph{strong order conditions},
\begin{equation}\label{eq:strongordercond}
\widetilde{J}_w(t_0+h;t_0)=J_w(t_0+h;t_0), \qquad w\in\W,\:\|w\| = 1/2, 1, 3/2,\dots, \mu,
\end{equation}
are \emph{sufficient} to guarantee strong order \(\geq \mu\). Under suitable assumptions on
\eqref{eq:stratonovich}, it may be proved that when the order conditions hold the local error actually
possesses a \(\mathcal{O}(h^{\mu+1/2})\) bound in the \(L^p\) norms, \(p<\infty\). Here our interest lies in
the combinatorial aspects of the theory and will not be concerned with the derivation of such bounds; the
interested reader is referred to \cite{alfonso}.

Are the strong  order conditions \eqref{eq:strongordercond} \emph{necessary} as well as sufficient to achieve strong order \(\geq \mu\)? This question may be
discussed in two different scenarios:

\begin{itemize}
\item \emph{Specific system.} In this case we are only interested in \eqref{eq:stratonovich} for a fixed,
    specific choice of dimension \(d\) and vector fields \(f_\ell\) in \(\R^d\).
\item \emph{General system.} Here \(\A\) and the coefficients \(\widetilde J_w\) are fixed and one demands that the
series \eqref{eq:strongordercond} only comprises terms of weight \(\geq \mu+1/2\) for each choice of
 \(d\) and each choice of vector fields \(f_\ell\), \(\ell\in\A\), in \eqref{eq:stratonovich}.
\end{itemize}

While the general system scenario is not without mathematical interest, in practice it is   the specific system case that matters. This point, that would be true for any numerical integrator, is especially so for splitting algorithms: one of the main advantages of the splitting idea is its versatility to  be tailored to the specific problem at hand.

In the specific system scenario is  possible that for some words \(w\) the  word basis functions \(f_w\)
vanish at each \(x_0\). If that is the case, it is not necessary to impose the order conditions \(\delta_w=0\)
associated with such words. This is illustrated in \cite{alfonso} in the case of the Langevin dynamics, whose
structure implies that many \(f_w\) vanish.

    In  the general system  scenario the conditions
\eqref{eq:strongordercond} \emph{are necessary} for strong order \(\geq \mu\), in view of the second item
of the lemma below that show that the word basis functions are independent.

\begin{lemma}\label{lem:independence}Fix the alphabet \(\A\) and choose \(w\in \W\), \(w\neq\emptyset\). There exist a value of the
dimension \(d\), vector fields \(f_\ell\), \(\ell\in\A\), in \(\R^d\), and a scalar observable \(\chi\), which depend on \(\A\) and \(w\), such that,
\begin{itemize}
\item \(D_w\chi(0)=1\) and \(D_u\chi(0)=0\) for each nonempty   \(u\in\W\), \(u\neq w\).
\item The first component \(f^1_u(0)\) of the vector \(f_u(0)\in\R^d\) vanishes for each nonempty \(u\in\W\), \(u\neq w\), while \(f^1_w(0)=1\).
\end{itemize}
\end{lemma}
\begin{proof} The first item follows from the second by choosing \(\chi\) to be the first coordinate mapping \(x\mapsto x^1\).

For the second item, the idea of the proof is best understood by means of an example. Suppose that \(w=\ell\ell m\ell m\) with \(\ell\neq m\). Then set \(d=5\),
\[f_\ell(x) = [0,x^3,0,x^5,1]^T,\qquad f_m(x) = [x^2,0,x^4,0,0]^T.
\]
(recall that superscripts denote components) and \(f_k(x)=0\) any remaining letters. Thus
\[
\sum_{k\in\A} f_k(x) = [x^2,x^3,x^4,x^5,1]^T.
\]
Because second and higher derivatives of the fields vanish, the recurrence \eqref{eq:recurrencebasis} shows that for any word \(u= k_n\dots k_1\)
\[
f_u(0) =  f^\prime_{k_1} \cdots  f^\prime_{k_{n-1}}f_{k_n}(0).
\]
 Assume that \(f_u^1(0)\neq 0\). The Jacobian matrix \( f_{k_1}^\prime\) must have a nonzero element in its first row and  this implies that \(k_1=m\). Then, by definition of \(f_m\), the first row is \([0,1,0,\dots,0]\), so that the second component of
 \[
 f^\prime_{k_2} \cdots  f^\prime_{k_{n-1}}f_{k_n}(0)
 \]
 must be nonzero. This implies that \(k_2=\ell\). By repeating this argument, we conclude that \(u=w\) and \(f^1_w(0)=1\).

 For a general word \(w\), things are as follows. The dimension \(d\) is taken equal to the number of letters in \(w\). The field
 \[
\sum_{k\in\A} f_k(x) = [x^2,x^3,\dots, x^d,1]^T
\]
is split in such a way that its \(d-j+1\), \(j=1,\dots, d\)  component is assigned to the field \(f_k\) if
\(k\) is the letter that occupies  the \(j\)-th position in \(w\).
(In this way there as many nonzero vector fields \(f_k\) as distinct letters in \(w\).)
\end{proof}

For \(\sigma \in\N\), \(\sigma >0\), the \emph{weak order conditions}
\begin{equation}\label{eq:weakordercond}
\E\big(\widetilde{J}_w(t_0+h;t_0)\big)=\E\big(J_w(t_0+h;t_0)\big), \qquad w\in\W,\:\|w\| =  1,2,\dots, \sigma,
\end{equation}
are \emph{sufficient} to ensure that the series in \eqref{eq:mainweak} only comprises terms of weight \(\geq \sigma+1\), or, as we shall say, the integrator has \emph{weak order} \(\geq \sigma\). In a general system scenario the weak order conditions are also \emph{necessary} in view of the first item in the preceding lemma.

\subsection{Ito order conditions}
For the Ito case the strong and week order conditions are
\begin{equation}\label{eq:strongordercondito}
\widetilde{I}_w(t_0+h;t_0)=I_w(t_0+h;t_0), \qquad w\in\overline{\W}_0,\:\|w\| =  1/2, 1,3/2,\dots, \mu,
\end{equation}
and
\begin{equation}\label{eq:weakordercondito}
\E\big(\widetilde{I}_w(t_0+h;t_0)\big)=\E\big(I_w(t_0+h;t_0)\big), \qquad w\in\overline \W,\quad\|w\| =  1,2,\dots,
\sigma,
\end{equation}
respectively. They guarantee that the series in \eqref{eq:mainstrongito} (respectively \eqref{eq:mainweakito}) only consists of terms of weight \(\geq\mu\) (respectively \(\geq \sigma\)), or, as we shall say,  the integrator has
\emph{strong order} \(\geq \mu\) (respectively \emph{weak order}  \(\geq \sigma\)).

It is possible to show (but the very long proof will not be reproduced here) that, in the general system scenario and if
\(\mu=1/2, 1, 3/2\), the conditions \eqref{eq:strongordercondito} are necessary to achieve strong order \(\geq \mu\). Similarly it may be proved that
\eqref{eq:weakordercondito} are necessary to have weak order \(\geq \sigma\) for general systems if \(\sigma = 1, 2, 3\). These particular values of \(\mu\) and \(\sigma\) are sufficient for establish the order barrier in Theorem~\ref{th:barrier} below.  We believe the strong are weak order conditions are  necessary for arbitrary \(\mu\) or \(\sigma\) but a proof is not yet available.

\subsection{Extensions}

In \eqref{eq:stratonovich} or \eqref{eq:ito} is assumed that all vector fields \(f_a\) and \(f_A\) are equally
important. In several applications this may not be the case. For instance, consider the system
\[
dx = f_a(x)\,dt+\epsilon f_b(x)\,dt+f_A(x)\circ d\B_A, \qquad 0<\epsilon \ll 1,
\]
or its Ito counterpart, where the split systems \(\{a,A\}\), \(\{b\}\) may be solved in closed form. Thus we
are dealing with small perturbation of an integrable system and it makes sense, when expanding the local
error, to track not only powers of \(h\) but also powers of \(\epsilon\), as in done in e.g.\ \cite{maka} in
the deterministic scenario. That task is easily accomplished with the tools presented so far. Details will not
be given.
\section{The shuffle and quasishuffle products}

The  conditions in \eqref{eq:strongordercond} are \emph{not independent}; for instance the  order condition
corresponding to the two-letter word \(\ell\ell\) is fulfilled whenever the order condition for \(\ell\) is
fulfilled. (Note that the dependence between order conditions and the necessity of the order conditions
discussed above are completely different issues.) Similarly there are dependencies within each of the set of
conditions \eqref{eq:weakordercond}, \eqref{eq:strongordercondito} and \eqref{eq:weakordercondito}. The study
of this issue requires the help of the shuffle and quasishuffle products. More generally, these products play
a key role when working in many developments involving  elements of \(\R\langle\langle \A\rangle\rangle\) or
\(\R\langle\langle \eA\rangle\rangle\) \cite{reut}. In the deterministic case the shuffle relations between
iterated integrated were first noted by Ree \cite{ree}. The stochastic scenario was addressed by Gaines
\cite{gaines}. On the other hand there is much literature relating the shuffle and quasishuffle products to
stochastic integration, see e.g.\ \cite{kur}.

 We begin with the Stratonovich/shuffle case.
The more complicated Ito/quasi\-shuffle case is presented later.

\subsection{The shuffle product}
To motivate the introduction of the shuffle product,
we begin by noting that if \(\varphi:\R^d\rightarrow\R^d\) is any mapping and \(\Phi\) the associated pullback operator, then, for any pair of
scalar-valued observables \(\chi_1,\chi_2\),
\[
\Phi(\chi_1\cdot\chi_2) = (\Phi\chi_1)\cdot(\Phi\chi_2),
\]
where \(\cdot\) denotes the standard (pointwise) product of observables,
i.e.\ \((\chi_1\cdot\chi_2)(x) = \chi_1(x)\chi_2(x)\) for \(x\in\R^d\). In other words \(\Phi\) is \emph{multiplicative}. The series of differential operators \(D_J\) and \(D_{\widetilde J}\) that expand the pullback operators associated with \(\varphi_{t_0+h;t_0}\) and \(\psi_{t_0+h;t_0}\)
are similarly multiplicative.
Now, it is easily checked that if \(S\in\R\langle\langle \A\rangle\rangle\), then, in general
\[
D_S(\chi_1\cdot\chi_2) \neq (D_S\chi_1)\cdot(D_S\chi_2)
\]
(for instance, if \(S=\ell\in\A\), then \(D_S(\chi_1\cdot\chi_2)= (D_S\chi_1)\cdot\chi_2+\chi_1\cdot(D_S\chi_2)\)). Therefore
the coefficients \(J_w\) and \(J_{\widetilde w}\) of the series \(D_J\) and \(D_{\widetilde J}\) must have some special property that tells them
 apart form \lq\lq general\rq\rq\ coefficients; as we shall see,  that property  explains
  the dependence between the order conditions.

In order to identify when a series \(D_S\) is multiplicative, we first investigate the action of a differential operator
\(D_w\), \(w\in\W\), on a product \(\chi_1\cdot\chi_2\). For instance, for \(k,\ell,m\in\A\), a trivial computation leads to \begin{eqnarray*}
D_{k\ell m}(\chi_1\cdot\chi_2)&=& (D_{k\ell m}\chi_1)\cdot(D_\emptyset \chi_2)
+(D_{k\ell}\chi_1)\cdot(D_m \chi_2)\\
&&
+ (D_{k m}\chi_1)\cdot (D_\ell \chi_2)
+ (D_{\ell m}\chi_1)\cdot(D_k \chi_2)\\
&&
+(D_k\chi_1)\cdot (D_{\ell m} \chi_2)
+(D_{\ell}\chi_1)\cdot (D_{km} \chi_2)\\
&&
+ (D_{m}\chi_1)\cdot (D_{k\ell} \chi_2)
+(D_\emptyset \chi_1)\cdot (D_{k\ell m}\chi_1).
\end{eqnarray*}
The right-hand side contains eight pairs of words \((k\ell m,\emptyset)\), \((k\ell,m)\), \dots\ What do these
pairs have in common? They are precisely the pairs such that when \emph{shuffled} give rise to the word
\(k\ell m\) in the left-hand side. By definition, the \emph{shuffle product} \(u\sh v\) of two words with
\(m\) and \(n\) letters is the sum of the \((m+n)!/(m!n!)\) words that may be formed by interleaving the
letters of \(u\) with those of \(v\) while keeping the letters in the same order as they appear in \(u\) and
\(v\). For instance \(k\ell\sh m= k\ell m+km\ell+\ell k m\), \( \ell\sh\ell = \ell+\ell = 2\ell\), etc. More
formally, the shuffle product of words may defined recursively by the relations \cite[Section 1.4]{reut}
\[
\emptyset \sh \emptyset = \emptyset,\qquad \emptyset \sh \ell = \ell \sh \emptyset = \ell,\quad \ell\in\A,
\]
and
\begin{equation}\label{eq:shufflerecurrence}
u\ell\sh vm =(u\ell \sh v)m+(u\sh vm)\ell,\qquad u,v\in\W,\quad \ell,m\in\A.
\end{equation}
The last equality corresponds to the fact that the words arising from shuffling \(u\ell\) and \(vm\) necessarily end with either the last letter of \(u\ell\) or the last letter of \(vm\).
Note that for words \(u,v\in\W\), the shuffle \(u\sh v\) is in general not a word but an element of the space \(\R\langle \A\rangle\) of linear combination of words. By linearity, the shuffle product may be trivially extended to a commutative, associative product in
\(\R\langle \A\rangle\); for instance \( (3k+\ell) \sh (\ell-m) = 3k\ell+3\ell k -3km-3mk+2 \ell \ell-\ell m-m\ell \).

At this stage we introduce some additional notation that will be used frequently below. If
\(S\in\R\langle\langle \A\rangle\rangle\) is a series and \(p=\sum_w p_ww\in\R\langle \A\rangle\), we set
\begin{equation}\label{eq:bilinear}
(S,p) = \sum_w S_w p_w;
\end{equation}
the sum is well defined because only a finite number of coefficients \(p_w\) are \(\neq 0\). In the case where
\(p\) coincides with a word \(w\), \((S,w)\) is just the coefficient \(S_w\); for general \(p\), \((S,w)\) is
a linear combination of coefficients \(S_w\). Obviously \((\cdot,\cdot)\) is a real-valued bilinear map. With
this notation, we may present the following result (that generalizes the formula for \(D_{k\ell
m}(\chi_1\cdot\chi_2)\) above).
\begin{prop}\label{prop:shuffle}For any \(S\in\R\langle\langle \A\rangle\rangle\) and any pair of observables
\[
D_S(\chi_1\cdot\chi_2) =  \sum_{u,v\in\W} (S,u\sh v)\: D_u\chi_1\cdot D_v\chi_2.
\]
\end{prop}
\begin{proof} It is sufficient to prove the case where \(S\) coincides with a word. The proof is by induction on the length (number of letters) of the word (not to be confused with its weight). When \(S\) is the empty word the result is trivial because, necessarily, in the right-hand side \((S,u\sh v)\) vanishes except if \(u=v=\emptyset\) when \((S,u\sh v)=1\). We assume that the result is true for the word \(w\) and prove it for the longer word \( w\ell\). Since \(D_\ell\) is a first-order differential operator we may write
\[
D_{w\ell}  (\chi_1\cdot\chi_2) =D_wD_\ell (\chi_1\cdot\chi_2)=  D_w (D_\ell \chi_1\cdot \chi_2+\chi_1\cdot D_\ell \chi_2),
\]
so that, by the induction hypothesis,
\[
D_{ w\ell }  (\chi_1\cdot\chi_2) =  \sum_{u,v\in\W} (w,u\sh v)\: \Big(D_{u\ell}\chi_1 \cdot D_v \chi_2 +
D_{u}\chi_1 \cdot D_{v\ell} \chi_2  \Big).
\]
Now from the definition of shuffle \((w,u\sh v)= (w\ell,u\ell\sh v) = (w\ell, u\sh v\ell)\) and therefore
\[
D_{ w\ell }  (\chi_1\cdot\chi_2) =  \sum_{u,v\in\W} (w\ell,u\ell\sh v)\: D_{u\ell}\chi_1 \cdot D_v \chi_2 + \sum_{u,v\in\W}
(w\ell,u\sh v\ell)
D_{u}\chi_1 \cdot D_{v\ell} \chi_2.
\]
The proof concludes by observing that, when \((w\ell,u^\prime \sh v^\prime)\) is \(\neq 0\), i.e.\ when \(w\ell\) is one of the words arising when shuffling \(u^\prime\) and \(v^\prime\),
the last letter in \(w\ell\) must be either  the last letter of \(u^\prime\) or the last letter of \(v^\prime\), so that either \(u^\prime=u\ell\) or \(v^\prime=v\ell\).
\end{proof}

Since, clearly
\[
(D_S\chi_1)\cdot(D_S\chi_2) = \sum_{u,v\in\W} (S,u)(S,v)\: D_u\chi_1\cdot D_v\chi_2,
\]
we may write
\begin{equation}\label{eq:23f}
D_{w\ell}  (\chi_1\cdot\chi_2) - (D_S\chi_1)\cdot(D_S\chi_2) = \sum_{u,v\in\W} \Big((S,u\sh v)-(S,u)(S,v)\Big)\: D_u\chi_1\cdot D_v\chi_2.
\end{equation}
This leads trivially to next result:

\begin{prop}\label{prop:multiplicative}Consider a series \(S\in\R\langle\langle \A\rangle\rangle\), \(S\neq 0\). The series of operators \(D_S\) is multiplicative
if  \(S_\emptyset=1\) and for each pair of words \(u,v\in\W\), the so-called shuffle relation
\[
(S,u\sh v) = (S,u)(S,v)
\]
holds.
\end{prop}

Thus the shuffle relations are equations that link the different coefficients \(S_w\), \(w\in\W\). For
instance, from the shuffle \(\ell\sh\ell= 2\ell\ell\), \(\ell\in\A\), we have the shuffle relation \(S_\ell^2
= 2S_{\ell\ell}\) and, from the shuffle \(k\sh \ell= k \ell+ \ell k \), \(S_kS_\ell = S_{k\ell}+S_{\ell k}\).

Proposition~\ref{prop:multiplicative} in tandem with the following result give a new proof of  the multiplicativity of \(D_J\) that we pointed out above.

\begin{prop}\label{prop:Jmultiplicative} The Stratonovich  iterated integrals \(J_w(t_0+h;t_0)\) satisfy the shuffle relations.
\end{prop}

\begin{proof}  For the shuffling of two  letters \(\ell, m\in \A\), the integration by
parts formula
\begin{eqnarray}\nonumber
&&\big(\B_\ell(t_0+h)-\B_\ell(t_0)\big)
\big(\B_m(t_0+h)-\B_m(t_0)\big)= \\
\nonumber &&\qquad \qquad\qquad \int_{t_0}^{t_0+h} \big(\B_\ell(t_0+s)-\B_\ell(t_0)\big)\circ d\B_m(s)\\
&&\qquad\qquad\qquad+\int_{t_0}^{t_0+h} \big(\B_m(t_0+s)-\B_m(t_0)\big)\circ d\B_\ell(s),
\label{eq:integrationbyparts}
\end{eqnarray}
is a statement of the shuffle relation \(J_\ell J_m = J_{m\ell}+J_{\ell m}\) (recall that if \(\ell\) or \(m\) are not stochastic, then
\(\B_\ell(t) = t\) or \(\B_m(t) = t\) respectively). General shuffles are dealt with by induction based on the
recursive definition of the shuffle product in \eqref{eq:shufflerecurrence} and the recursion
\eqref{eq:recurrenceintegrals} for the iterated integrals.
\end{proof}

To present a similar result for the integrator we need a lemma:
\begin{lemma}\label{lem:shuffle}Let \(S,T\in\R\langle\langle\A\rangle\rangle\), with \(S_\emptyset = T_\emptyset =1\), satisfy the
shuffle relations. Then the product \(ST\) has \((ST)_\emptyset = 1\) and satisfies the shuffle relations.
\end{lemma}
\begin{proof} Recall that the coefficients of \(ST\) are given by the convolution product as in
\eqref{eq:convolution}, which is based on deconcatenation. The result is a consequence of the following
observation: the deconcatenation of the words in a shuffle \(u\sh v\) may be found by shuffling the
deconcatenations of \(u\) and \(v\). An example of this observation follows. Deconcatenation of the shuffle
\(k\ell\sh m= k\ell m+k m \ell+mk\ell\) gives the 12 pairs
 \[(k\ell
m,\emptyset)+(k\ell ,m)+(k ,\ell m)+(\emptyset,k\ell m)+(km\ell,\emptyset)+\cdots+ (\emptyset,mk\ell). \] On
the other hand by deconcatenating \(k\ell\) we obtain  \((k\ell,\emptyset)+(k,\ell)+(\emptyset,k\ell)\), and
by deconcatenating \(m\) obtain  \((m,\emptyset)+(\emptyset,m)\). Shuffling now as in
\begin{eqnarray*}
(k\ell,\emptyset)\sh (m,\emptyset) &=& (k\ell\sh m, \emptyset \sh \emptyset) \\
 (k\ell,\emptyset)\sh (\emptyset,m) &=& (k\ell\sh \emptyset,\emptyset\sh m) ,
\end{eqnarray*}
etc.\ yields the same 12 pairs (the first line of the display gives  \((k\ell m+k m \ell+mk\ell,\emptyset)\),
the second \((k\ell, m)\), etc). To prove the observation in the general case, one may use the recurrence
\eqref{eq:shufflerecurrence}.

By using the observation, \((ST,u\sh v)\) may be written as a sum of products
\[\sum_{ij} (S,u_i\sh
v_j)(T,u_i^\prime\sh v_j^\prime),
\]
(\(u_iu_i^\prime = u\) and \(v_jv_j^\prime = v\)) or, since \(S\) and  \(T\) satisy the shuffle relations,
\begin{eqnarray*}&&\sum_{ij} (S,u_i)(S,v_j)(T,u_i^\prime)(T,v_j^\prime)\\
&&\qquad\qquad = \sum_i (S,u_i)(T,u^\prime_i) \sum_j
(S,v_j)(T,v^\prime_j)=(ST,u)(ST,v).
\end{eqnarray*}
\end{proof}

\begin{prop}\label{prop:Jtildemultiplicative} For a splitting integrator for the Stratonovich system \eqref{eq:stratonovich} the coefficients \(\widetilde{J}_w(t_0+h;t_0)\) satisfy the shuffle relations.
\end{prop}

\begin{proof}
The proof  is a trivial consequence of the lemma, because  \(D_{\widetilde J}\) is a composition of solution
operators \(D_{J_i}\) associated with  the split systems and therefore, by the preceding proposition, a
composition of operators that satisfy the shuffle conditions.
\end{proof}

After the last two propositions, it is easy to see that the Stratonovich strong order conditions are not
independent. For instance from the shuffle relations
 \(J_\ell(t_0+h;t_0)^2 = 2 J_{\ell\ell}(t_0+h;t_0)\) and \(\widetilde J_\ell(t_0+h;t_0)^2 = 2 \widetilde J_{\ell\ell}(t_0+h;t_0)\), we conclude that the strong
order condition \(\widetilde J_{\ell\ell}(t_0+h;t_0) = J_{\ell\ell}(t_0+h;t_0)\) corresponding to the word
\(\ell\ell\) is fulfilled if the strong order condition \(\widetilde J_{\ell}(t_0+h;t_0) =
J_{\ell}(t_0+h;t_0)\) holds. Analogously, if \(k\neq \ell\) the order condition for \(k\ell\) is implied by
those of \(\ell k\), \(k\) and \(\ell\), etc. It is possible to obtain \emph{independent} order conditions by
keeping only the conditions corresponding to the so-called \emph{Lyndon} words \cite{gaines} that we describe next. We order the alphabet \(\A\) and then order words lexicographically; a Lyndon
word is a word that is strictly smaller than all the words obtained by rotating its letters. If the alphabet
is \(\A=\{a,A\}\) and \(a<A\), then \(aA\) is a Lyndon word because it precedes the rotated \(Aa\). Similarly
\(aaA\) is a Lyndon word while \(aAa\) and \(Aaa\) are not. For this simple alphabet, the Lyndon words with
three or fewer letters are \(a\), \(A\), \(aA\), \(aaA\), \(aAA\); their order conditions are independent and
imply, via the shuffle relations, the order conditions for \(aa\), \(AA\), \(aaa\), \(aAa\), \(Aaa\), \(AaA\),
\(AAa\) and \(AAA\).

For reasons of brevity, the independence of the Stratonovich weak order
conditions will not be discussed in this paper.

\begin{rem}From \eqref{eq:23f} and Lemma~\ref{lem:independence} the shuffle conditions
 are \emph{necessary} for \(D_S\) to be multiplicative for each choice of \(d\) and
 vector fields \(f_\ell\), \(\ell\in\A\). Hence the last two propositions may be proved in an alternative way:
 one first notices the multiplicativity of \(D_J\) and \(D_{\widetilde J}\) as expansions of pullbacks associated
 with the true and numerical solution respectively and then concludes that the shuffle conditions are satisfied
 because the multiplicativity holds for all choices of vector fields. Recall from Remark~\ref{rem:integralsbrownian}
 that changing the vector fields does not alter the iterated integrals.
 \end{rem}

\subsection{The quasishuffle product}
As we noted above, the shuffle property of the Stratonovich  iterated integrals stems from the formula of integration by parts in
\eqref{eq:integrationbyparts}.
For the Ito calculus,  formula \eqref{eq:integrationbyparts} has to be replaced by
\begin{eqnarray}
\nonumber &&\big(\B_\ell(t_0+h)-\B_\ell(t_0)\big)
\big(\B_m(t_0+h)-\B_m(t_0)\big)= \\
\nonumber &&\qquad \qquad\qquad \int_{t_0}^{t_0+h} \big(\B_\ell(t_0+s)-\B_\ell(t_0)\big) d\B_m(s)\\
\nonumber &&\qquad\qquad\qquad+\int_{t_0}^{t_0+h} \big(\B_m(t_0+s)-\B_m(t_0)\big) d\B_\ell(s)\\
&& \qquad\qquad\qquad+\left[\big(\B_\ell(t_0+h)-\B_\ell(t_0)\big),
\big(\B_m(t_0+h)-\B_m(t_0)\big)\right],\label{eq:integrationbypartsito}
\end{eqnarray}
where  the last term represents the quadratic covariation (see e.g.\ \cite[Chapter 5]{b}). If \(\ell=m\in\Asto\), then the quadratic covariation
in \eqref{eq:integrationbypartsito} is
\(h\); for all other combinations of letters the quadratic covariation vanishes.

The \emph{quasishuffle product} \(\bowtie \) to be defined presently is such  that for any two letters \(\ell,
m\in \eA\), the computation of \(\ell\bowtie m\) mimics the integration by parts relation
\eqref{eq:integrationbypartsito}. In combinatorial algebra, the definition of a quasishuffle product depends
on the choice of a so-called bracket \([\cdot,\cdot]\); different brackets lead to different quasishuffle
products as defined by Hoffman \cite{hoffman}. Throughout this paper we only work with one fixed choice of
bracket defined as follows.
 For letters \(\ell,m\in\eA\),  \([\ell,m]\) takes the value \(\bar A\in \R\langle
\eA\rangle\) if \(\ell=m=A\in\Asto\); \([\ell,m]=0\in \R\langle \eA\rangle\) in all other cases. Then the
quasishuffle product of words \(u\bowtie v\in \R\langle \eA\rangle\) is  defined recursively by
\[
\emptyset \bowtie \emptyset = \emptyset,\qquad \emptyset \bowtie \ell = \ell \bowtie \emptyset = \ell,\quad \ell\in\eA,
\]
and
\[
u\ell\bowtie vm =(u\ell \bowtie v)m+(u\bowtie vm)\ell+(u\bowtie v)[x,y],\qquad u,v\in\W,\quad \ell,m\in\eA.
\]
In the particular case \(u = v=\emptyset\), the last relation yields \(\ell\bowtie m= \ell m + m\ell +[\ell,m]\),
 a transcription of  \eqref{eq:integrationbypartsito}.

The next four results are counterparts of Propositions~\ref{prop:shuffle}--\ref{prop:Jtildemultiplicative}. The bilinear form \((\cdot,\cdot)\) in \eqref{eq:bilinear}, which we defined in \(\R\langle\langle \A\rangle\rangle\times \R\langle \A\rangle\), is now extended to \(\R\langle\langle \eA\rangle\rangle\times \R\langle \eA\rangle\).
\begin{prop}\label{prop:quasishuffle}For any \(S\in\R\langle\langle \eA\rangle\rangle\) and any pair of observables
\[
D_S(\chi_1\cdot\chi_2) =  \sum_{u,v\in\overline\W} (S,u\bowtie v)\: D_u\chi_1\cdot D_v\chi_2.
\]
\end{prop}
\begin{proof} One may use the same technique as in Proposition~\ref{prop:shuffle}. Here the proof is lengthier because it has to contemplate the possibility  \(\ell = \bar A\), \(A\in\Asto\) in which case
\(D_\ell\) is a second order operator.
\end{proof}

This yields immediately:

\begin{prop}\label{prop:multiplicativeito} Consider a series \(S\in\R\langle\langle \eA\rangle\rangle\), \(S\neq 0\). Then the series of operators \(D_S\) is multiplicative
if  \(S_\emptyset=1\) and for each pair of words \(u,v\in\W\), the  quasishuffle relation
\[
(S,u\bowtie v) = (S,u)(S,v)
\]
holds.
\end{prop}

The proofs of the following propositions are similar to those of Propositions~\ref{prop:Jmultiplicative} and
\ref{prop:Jtildemultiplicative} respectively.

\begin{prop}\label{prop:Imultiplicative} The  the Ito iterated  integrals \(I_w(t_0+h;t_0)\) satisfy the quasishuffle relations.
\end{prop}

\begin{prop}\label{prop:Itildemultiplicative} For a splitting integrator for the Ito system \eqref{eq:ito}, the coefficients \(\widetilde{I}_w(t_0+h;t_0)\)
 satisfy the quasishuffle relations.
\end{prop}

The last two propositions show immediately that the  Ito strong order conditions are not independent. The dependence between the Ito weak order conditions will be discussed after Proposition~\ref{eq:shuffleexpnumerico}.

\subsection{Concatenating Chen series}
The shuffle (quasishuffle) relations constrain the values of Stratonovich (Ito) iterated integrals corresponding to different words but based on a common interval \((t_0,t_0+h)\). Iterated integrals corresponding to adjacent intervals are also interrelated, as we now discuss.

 Solution operators of Stratonovich or Ito systems satisfy
\[
\varphi_{t_2,t_1}\circ\varphi_{t_1,t_0} = \varphi_{t_2,t_0},\qquad t_2\geq t_1\geq t_0.
\]
From here we get the following relations between series of operators
\[
D_{J(t_1;t_0)}D_{J(t_2;t_1)} = D_{J(t_2;t_0)}\qquad
D_{I(t_1;t_0)}D_{I(t_2;t_1)} = D_{I(t_2;t_0)},\qquad t_2\geq t_1\geq t_0;
\]
the corresponding relations between elements of \(\R\langle\langle \A\rangle\rangle\) or \(\R\langle\langle
\eA\rangle\rangle\) (Chen series) are
\begin{equation}\label{eq:chen}
J(t_1;t_0)J(t_2;t_1) = J(t_2;t_0);\qquad
I(t_1;t_0)I(t_2;t_1) = I(t_2;t_0),\qquad t_2\geq t_1\geq t_0.
\end{equation}

The equalities in \eqref{eq:chen} are, in view of \eqref{eq:convolution}, a family of relations between
iterated integrals first noted by Chen \cite{chen} in the case where there are no stochastic letter. For
instance, for words with two letters:
\[
J_{\ell m}(t_2;t_0)^2 = J_{\ell m}(t_1;t_0)+J_{\ell}(t_1;t_0)J_m(t_2;t_1)+J_{\ell m}(t_2;t_1),
\]
etc. These relations may alternatively be proved by  manipulating the integrals, without going through the series of differential operators
as above.
\subsection{Composing word series}
We conclude our study of the shuffle and quasishuffle products by  showing that, in some circumstances, the composition \(\W_T(
\W_S(x) )\) of two word series is  another word series.

Let us begin with the Stratonovich case. If \(\chi\) is an observable and \(w\in\W\), then \(D_w\chi\) is a
sum of terms each of which is a derivative \(\chi^{(s)}(x)\) acting on combinations of derivatives of the
functions \(f_k\), \(k\in\A\). A simple example is:
\[
D_{\ell m}\chi(x) =\chi^{\prime\prime}(x)[f_\ell(x),f_m(x)]+ \chi^\prime(x) f^\prime_m(x)f_\ell(x).
\]
Here, the word \(\ell m\) may have weight \(1\), \(3/2\) or \(2\) depending of whether \(\ell\) and \(m\) are
stochastic of deterministic; the thing to observe is that in each term of the right-hand side of the last
equality the \(f_k\) \(k\in\A\), that appear have a combined weight that matches the weight of \(\ell m\).

If  \(S\in\R\langle\langle \A\rangle\rangle\) and \(D_S\) is the corresponding series of differential
operators we may arrange \(D_S\chi\) by grouping the terms where the combined weight of the \(f_k\) that
appear is successively 0, 1/2, 1, 3/2, etc. On the other hand if \(\W_S(x)\) is the associated word series and
\(S_\emptyset = 1\) so that \(W_S(x)-x = \mathcal{O}(1/2)\), we may Taylor expand as follows
 \begin{eqnarray*}
&& \chi(\W_S(x)) = \chi(x+[W_S(x)-x]) = \chi(x)+ \chi^\prime(x) [W_S(x)-x]\\
&&\qquad\qquad\qquad\qquad\qquad+\frac{1}{2}\chi^{\prime\prime}(x)
 [W_S(x)-x,W_S(x)-x]+\cdots
 \end{eqnarray*}
Here the right-hand side may be arranged, as we did in the case of \(D_S\chi\), by grouping the terms where the
combined weight of the \(f_k\) that appear is successively 0, 1/2, 1, 3/2, etc. This arrangement may be
carried out because \([W_S(x)-x]^r\) only contributes terms of combined weight \(\geq r/2\) and therefore
for each weight there is only a finite number of terms to be grouped. It turns out that if \(S\) is
multiplicative the expansions of \(D_S \chi(x)\) and \(\chi(W_S(x))\) coincide.

\begin{prop}Suppose that \(S\in\R\langle\langle \A\rangle\rangle\) has \(S_\emptyset = 1\) and satisfies
the shuffle relations. Then for any observable \(\chi\), the expansion of \(\chi(W_S(x))\) coincides with
\(D_S \chi(x)\).
\end{prop}
\begin{proof} If \(\chi\) is one of the coordinate mappings \(x\mapsto x^i\), then the result
is true because, by definiton, the \(i\)-th component of the word-basis function \(f_w\) is obtained by
applying \(D_w\) to the \(i\)-th coordinate mapping. If \(\chi\) is a product of coordinate mappings, the
result holds because \(D_S\) acts multiplicatively. By linearity the result is true if \(\chi\) is a
polynomial. Then the result hold for smooth \(\chi\) because it holds for the Taylor polynomials of any degree
of \(\chi\) around any base point \(x\).
\end{proof}

As a direct consequence we may state:

\begin{prop}\label{prop:composition}
Suppose that \(S\in\R\langle\langle \A\rangle\rangle\) has \(S_\emptyset = 1\) and satisfies
the shuffle relations. Then for any \(T\in\R\langle\langle \A\rangle\rangle\), \(\W_T( \W_S(x) )\) coincides
with the words series \(W_{ST}(x)\).
\end{prop}

\begin{proof} It is enough to note that, for each word basis function \(f_w(x)= D_w id(x)\), according to the preceding proposition,
  \(f_w(W_S(x))\) has the expansion \(D_S f_w(x)=D_SD_w id(x)\).
\end{proof}

The Ito case is completely parallel; the only change is that \( S\in\R\langle\langle \eA\rangle\rangle \) has to
be demanded to satisfy the quasishuffle relations rather than the shuffle relations.

In fact the computations leading to \eqref{eq:localerrorexample} or \eqref{eq:localerrorexampleito} are instances of the composition just described.

\section{Infinitesimal generators}

It is well known that the infinitesimal generators of \eqref{eq:stratonovich} or \eqref{eq:ito} play an
important role in the study of these systems, see e.g.\ \cite[Section 2.5]{pav}. In this section those generators are described in the language
of words. The material has an important implications for  the weak order conditions. We begin with Ito systems.

\subsection{The Ito generator}
For  system \eqref{eq:ito}, we consider the linear combination of \emph{deterministic} letters
\[
\mathfrak{G}= \sum_{\ell\in\eAdet}\ell = \sum_{a\in\Adet} a+\sum_{A\in\Asto} \bar A
\]
and define the exponential \(\exp(h\mathfrak{G})\in\R\langle\langle\eA\rangle\rangle\), \(h\in\R\), as the
series
\[
\emptyset + h \mathfrak{G} +\frac{h^2}{2} \mathfrak{G}^2+\cdots,
\]
where the powers are based on concatenation, e.g.
\[
\mathfrak{G}\mathfrak{G} = \sum_{a,b\in\Adet} ab +\sum_{a\in\Adet,B\in\Asto} a\bar B
+\sum_{A\in\Asto,b\in\Adet} \bar A b+\sum_{A,B\in\Asto} \bar A \bar B
\]
(note that the right-hand side is simply the sum all the words consisting of two deterministic letters from
\(\eA\)). The operator \(D_\mathfrak{G}\) is the \emph{infinitesimal generator} of \eqref{eq:ito},  a linear combination of first and
second order differential operators.
\begin{prop}\label{prop:generator}
The expectations of the Ito iterated integrals are given by \(\E I_w(t_0+h;t_0) = 0\) if \(w\in\overline \W\) has at
least one stochastic letter and \(\E I_w(t_0+h;t_0) = h^n/n!\) if \(w\in\overline \W\) consists of \(n\) deterministic
letters.

The following relation holds:
\[
\E I(t_0+h;t_0) = \exp(h \mathfrak{G}).
\]

For any observable and \(h>0\),
\[
\E\chi(x(t_0+h)) = \exp(h D_\mathfrak{G}) \chi(x_0),
\]
where  \(x(t)\) solves \eqref{eq:ito} with \(x(t_0)=x_0\) and the expectation is conditional on \(x_0\).
\end{prop}
\begin{proof}
For the first claim we recall that the expectation of Ito integrals vanishes. In addition it is trivially
computed that, when all the letters in a word  are deterministic,
\(I_w(t_0+h;t_0)=h^n/n!\), where \(n\) represents the number of letters.

By expanding \(\exp(h\mathfrak{G}\) as a series, one sees that the second claim is just a reformulation of the first.   An alternative proof of this second claim is as follows. As noted before
(Proposition~\ref{prop:iteratedintegralsito}), the distribution of the random variable \(I(t_0+h;t_0)\) is independent of \(t_0\)
and therefore we may write \(\E I(t_0+h;t_0)= \E I(h)\). The functions  \(\exp(h\mathfrak{G})\) and \(E I(h)\)
coincide at \(h=0\), where they take the common value \(\emptyset\). By taking expectations in
\eqref{eq:chen}, we find the semigroup relation \(\E(h_1+h_2) = \E I(h_1) \,\E I(h_2) \) for \(h_1,h_2\geq
0\). Differentiating with respect to \(h_1\) and then setting \(h_1=0\), \(h_2=h\) yields the linear, constant
coefficient differential equation \((d/dt)\E(h) =[(d/dh)\E I(0)]\,\E I(h) \).\footnote{This differential
equation in \(\R\langle\langle\eA\rangle\rangle\) is of course just a system of differential equations for the
coefficients \(\E I_w(h)\), \(w\in\overline \W\),  that presents no technical difficulty.}
On the other hand, a straightforward computation leads to \((d/dh)\exp(h\mathfrak{G}) = \mathfrak{G}
\exp(h\mathfrak{G})\), and the proof of the second statement concludes by noting that \((d/dh)\E
I(0)=\mathfrak{G}\) since \(\E I_w(h) = o(h)\) as \(h\downarrow 0\) if  \(w\) has length \(>1\) and all its
letters are deterministic.

For the last claim, just take expectations in  \eqref{eq:taylor2ito}.
\end{proof}

\begin{rem}
\label{rem:moments} The preceding proposition and the quasishuffle relations among the \(I_w\)
(Proposition~\ref{prop:Imultiplicative}) make it possible to compute all the \emph{moments} of the Ito
iterated integrals, as first suggested by Gaines  \cite{gaines}. The easiest example is given by the relation
\(A\bowtie A= 2AA+\bar A\) that leads to \(I_A^2 = 2I_{AA}+I_{\bar A}\); according to the proposition the
expectation of the right-hand side equals \(0+h\) and therefore  \(\E I_A^2= h\), a well known property of the
Brownian increment \(I_A\). The values of \(\E I_\ell^i\), \(\E(I_\ell^i I_{\ell m}^j)\), \(\ell,m \in \eA\),
\(i,j\in\N\), etc. may be computed similarly after writing the appropriate quasishuffles.
\end{rem}

\begin{prop}\label{prop:expectationIIto}
The expectations \(\E I_w(t_0+h;t_0)\), \(w\in\overline\W\) of the Ito iterated integrals satisfy the shuffle
relations.
\end{prop}
\begin{proof} This result is an easy consequence of the Proposition~\ref{prop:generator}.
With the abbreviation \(S=\exp(h \mathfrak{G})\), \((S,u)(S,v)\) and \((S,u\sh v)\) are both \(0\) if either
\(u\) or \(v\) have a stochastic letter. In other case, if \(u\) has \(m\) letters and \(v\) has \(n\),
\((S,u) =h^m/m!\), \((S,v) =h^n/n!\)  while \((S,u\sh v)\) is a sum of \((m+n)!/(m!n!)\) coefficients \(S_w\)
each of them with value \(h^{m+n}/((m+n)!\).
\end{proof}
\subsection{Weak order conditions in the Ito case}

We now turn to the series of expectations associated with a splitting integrator specified by the pullback series
\[
\widetilde I(t_0+h;t_0) = I^{(1)}(t_0+d_1h; t_0+c_1h)\cdots I^{(m)}(t_0+d_mh; t_0+c_mh).
\]

In general, the equality
\begin{equation}\label{eq:hypothesis}
\E \widetilde I(t_0+h;t_0) = \E I^{(1)}(t_0+d_1h;
t_0+c_1h)\cdots \E I^{(m)}(t_0+d_mh; t_0+c_mh)
\end{equation}
will \emph{not} hold because the \(I^{(i)}(t_0+d_ih; t_0+c_ih)\) are \emph{not independent}. However, as it
will shortly become clear, \eqref{eq:hypothesis} will typically be satisfied. We first present some examples
that will help to understand the situation.

Assume  that the alphabet \(\A\) consists of two letters \(a\) and \(A\). Choose a partition of the interval
\([0,1]\)
\[
0=c_1^\prime < d_1^\prime=c_2^\prime < d_2^\prime = c_3^\prime < \cdots
< d_{\nu-1}^\prime=c_{\nu}^\prime<d_\nu^\prime=1
\]
and let \(f_A\) act in the intervals \([t_0+c_i^\prime h,t_0+d_i^\prime h]\), while the deterministic \(f_a\) may
act on any set of intervals. In this case \eqref{eq:hypothesis} holds because the Brownian motion \(\B_A\)
acts on nonoverlapping intervals. This example may be easily extended to the case where there are additional
deterministic fields \(f_b\), \(f_c\), \dots; in the split systems some of them could be grouped with \(f_a\)
and some of them grouped with \(f_A\).

As a second example, assume that \(\A = \{A,B\}\) and use Strang's splitting with \(f_A\) acting first. Here
\(I^{(1)}\) and \(I^{(3)}\) are independent because  their intervals do not overlap, while the pairs
\(I^{(1)}\), \(I^{(2)}\) and \(I^{(2)}\), \(I^{(3)}\) are independent because they use independent Brownian
motions. Again this example may be easily generalized by adding additional deterministic and/or stochastic
letters.

We have the following general result:
\begin{lemma}\label{lem:abril}
Assume  that \(\Asto\neq \emptyset\) so that \eqref{eq:ito} does not degenerate into a deterministic
differential equations. If a splitting integrator for \eqref{eq:ito} has strong order \(> 0\) (i.e. \(\geq 1/2\), then
\eqref{eq:hypothesis} holds.
\end{lemma}
\begin{proof} As noted at the end of
Section~\ref{sec:ordercond},  the Ito strong order conditions with \(\mu= 1/2\) must be satisfied. Now for
each \(A\in\Asto\) the strong order condition corresponding to  \(A\), shows that
\(\sum_jI_A(t_0+d_{i_j}h;t_0+c_{i_j})=I_A(t_0+h;t_0)\), where the sum is extended to  all partial systems that
include \(f_A\). This implies that, for each fixed \(A\), the corresponding intervals \([c_{i_j},d_{i_j}]\)
cover the interval \([0,1]\) and therefore cannot overlap.
\end{proof}

Schemes that satisfy  \eqref{eq:hypothesis} have the special properties that we study next. To begin with,
 Lemma~\ref{lem:shuffle} clearly implies:
\begin{prop}\label{eq:shuffleexpnumerico}For splitting integrators for \eqref{eq:ito} that satisfy \eqref{eq:hypothesis}, the expectations coefficients
\(\E\widetilde I_w(t_0+h;t_0)\) satisfy the shuffle conditions.
\end{prop}

In turn this result and Proposition~\ref{prop:expectationIIto} show that the weak order conditions are not
independent when \eqref{eq:hypothesis} holds. For instance the weak order condition for \(\ell\ell\) is
implied by the weak order condition for \(\ell\in\overline W\), since, as noted repeatedly,
\(\ell\sh\ell=2\ell\ell\).

In the next proposition we need  the \emph{deterministic} system:
\begin{equation}\label{eq:deter}
dx =\sum_{a\in\Adet} f_a(x)\, dt+\sum_{A\in\Asto} f_A(x)\, dt,
\end{equation}
obtained by replacing the differentials \(d\B_A\) in the Ito system \eqref{eq:ito} by \(dt\). It is clear that
each splitting algorithm for \eqref{eq:ito} defines a splitting algorithm for \eqref{eq:deter} and vice versa.
\begin{prop} \label{prop:sameorder}
For splitting integrators for \eqref{eq:ito} that satisfy \eqref{eq:hypothesis}  and in the general
 vector field scenario, the following properties are equivalent:
 \begin{itemize}
\item The weak order  conditions \eqref{eq:weakordercondito}  hold  for a positive integer \(\sigma\).
\item When applied to the deterministic system \eqref{eq:deter}, the integrator has local error \(\mathcal{O}(h^{\sigma+1})\).
\end{itemize}
\end{prop}
\begin{proof} From \eqref{eq:hypothesis} and Proposition~\ref{prop:generator}
\[
\E \widetilde I(t_0+h;t_0) = \exp\Big(h (d_1-c_1) \mathfrak{G}^{(1)}\Big)\cdots
\exp\Big(h (d_m-c_m) \mathfrak{G}^{(m)}\Big),
\]
where the \(\mathfrak{G}^{(i)}\) are the generators of the partial systems and therefore sums of deterministic
letters. Condition \eqref{eq:weakordercondito}, requires that, in the series in
the last display, the terms corresponding to words with \(\leq \sigma\) letters coincide with those of
\[
\E I(t_0+h;t_0) = \exp(h \mathfrak{G}).
\]

To study the order for \eqref{eq:deter} we may also use words seeing a deterministic system as the particular
case of Ito system where there is no stochastic letter. If we denote by \(\bar A\) the (deterministic) letter
associated with the field \(f_A\), we then have
\[
\widetilde I(t_0+h;t_0) = \exp\Big(h (d_1-c_1) \mathfrak{G}^{(1)}\Big)\cdots
\exp\Big(h (d_m-c_m) \mathfrak{G}^{(m)}\Big),
\]
and
\[
I(t_0+h;t_0) = \exp(h \mathfrak{G}),
\]
and  order \(\sigma\) requires that the terms involving words with \(\sigma\) or fewer letters in the series
in the last two displays coincide.
\end{proof}

The following counterexample shows that, in the last two propositions,  hypothesis \eqref{eq:hypothesis}
cannot be dispensed with. For the alphabet \(\A = \{a,A\}\), we consider the integrator
\[
\varphi^{(A)}_{t_0+h/2,t_0}\circ \varphi^{(a)}_{t_0+h,t_0}\circ \varphi^{(A)}_{t_0+h/2,t_0}.
\]
While this is admittedly a contrived example, using the interval \([t_0,t_0+h/2]\) to finish the step (rather
than the more natural \([t_0+h/2,t_0+h]\)) may have some appeal. On the one hand the distribution of the
iterated integrals in \([t_0,t_0+h/2]\) is the same as that in \([t_0+h/2,t_0+h]\)) and, on the other hand,
working twice with \([t_0,t_0+h/2]\) may make it  possible to reuse Brownian increments. For this integrator
the hypothesis \eqref{eq:hypothesis} does not hold. A simple computation, similar to that preceding
\eqref{eq:localerrorexampleito}, yields
\[\widetilde I(t_0+h;t_0) = 1\emptyset +2 I_A A+ I_aa+ [I_A^2+2I_{AA}]AA+I_{\bar A}\bar A+\mathcal{O}(3/2)
\]
(the iterated integrals in the right-hand side are over \([t_0,t_0+h/2]\)). We note in relation with
Lemma~\ref{lem:abril} that here the order condition for \(A\) is obviously not satisfied. Taking expectations
in the last display,
\[
\E \widetilde I(t_0+h;t_0) = 1\emptyset+ 0A+ha+\frac{h}{2}AA+h\bar A+\mathcal{O}(2).
\]
Since \(0^2\neq 2\times h/2\), for the expectations, the shuffle relation corresponding to \(A\sh A = 2AA\)
does not hold. On the other hand, from Proposition~\ref{prop:generator},
\[
\E I(t_0+h;t_0) = 1\emptyset + 0 A+ha+ 0AA+h \bar A+\mathcal{O}(2)
\]
so that weak order conditions for \(\sigma = 1\) are \emph{not} satisfied. In the deterministic case the algorithm coincides with
Strang's splitting with local errors \(\mathcal{O}(h^3)\) (i.e. \(\sigma =2\)). Thus the weak order does not coincide with the
deterministic order.

 It turns out that, in the general system scenario, under \eqref{eq:hypothesis}, there is an order barrier:
 the weak order
 cannot be better than \(\sigma =2\).

\begin{theo}\label{th:barrier}
Assume that \eqref{eq:hypothesis} holds. There is no splitting integrator for \eqref{eq:ito} with weak order
\(\sigma \geq 3\).
\end{theo}
\begin{proof}  By contradiction. As noted at the end of Section~\ref{sec:ordercond},
 the Ito weak conditions with \(\sigma =3\) holds. From Proposition~\ref{prop:sameorder} the algorithm
 is of order \(\geq 3\) for deterministic problems,  which is known to be contradictory with the
condition  \(c_{i_j} < d_{i_j}\) \cite{order barrier}.
\end{proof}

\begin{rem}In the deterministic case this order barrier may be overcome by using complex coefficients; a full discussion of the relevant literature may be seen in \cite[Section 6.3.3]{blacalibro}. To our best knowledge complex coefficients
have not yet been tested in the stochastic scenario.
\end{rem}

\subsection{The Stratonovich generator}

We briefly outline how the preceding material has to be modified in the Stratonovich case. The expression for the generator is
\[
\mathfrak{G} = \sum_{a\in\Adet} a+\frac{1}{2} \sum_{A\in\Asto}AA\in\R\langle \A\rangle,
\]
and, in analogy with Proposition~\ref{prop:generator}, we have
\begin{equation}\label{eq:bk}
\E J(t_0+h;t_0) = \exp(h \mathfrak{G}),
\end{equation}
a formula that may be proved by showing, as in the Ito case, that the left- and right-hand sides satisfy the same initial value problem. As a consequence, one obtains the following formula for the expectation of observables:
\[
\E\chi(x(t_0+h)) = \exp(h D_\mathfrak{G}) \chi(x_0).
\]

Taking the coefficient of the word \(w\in\W\) in \eqref{eq:bk} gives the value of the expectations of the iterated integrals. Clearly \(\E J_w(t_0+h;t_0) = 0\) if \(w\) is not a
concatenation of deterministic letters \(a\in\Adet\) and pairs \(AA\), \(A\in\Asto\) (examples include \(AAA\) or \(ABAB\) if \(A\neq B\)). When \(w\) is such a concatenation, it is easily shown that
\[
\E J_w(t_0+h;t_0)  = \frac{1}{2^{\pi(w)}}\frac{h^{\| w\|}}{\| w\|!}
\]
where \(\pi(w)\) is the number of pairs that enter in the concatenation (for instance for \(AAaBBAA\), \(\pi=3\) and for \(AAAA\), \(\pi =2\)).
 Once the expectations \(\E J_w(t_0+h;t_0)\) are known, the shuffle relations in
 Proposition~\ref{prop:multiplicative} may be used to compute higher \emph{moments} of the iterated integrals,
 similarly to what was explained in Remark~\ref{rem:moments}.

As distinct from the \(\E I_w\), \(w\in\overline \W\), studied in Proposition~\ref{prop:expectationIIto}, the
 \(\E J_w\), \(w\in\W\), do not satisfy the shuffle relations
 (except of course in the degenerate  case where \(\Asto=\emptyset\)).

For integrators that satisfy the obvious analogue of \eqref{eq:hypothesis}, Proposition~\ref{prop:sameorder}
also holds in the Stratonovich case and therefore the order barrier in Theorem~\ref{th:barrier} also applies to the Stratonov\-ich interpretation.

\section{Relating the Stratonovich and Ito interpretations}
\label{sec:is}

In this paper, the Stratonovich and Ito theories have been developed in parallel.
It is well known that it is actually possible to map one into the other and we now present how
to do so by means of words.

\subsection{Relating the Stratonovich and Ito iterated integrals}

Along with the extended alphabet \(\eA\) that we used to carry out the Ito-Taylor expansion, let us now consider a new alphabet \(\A^\star\) that consists of all the deterministic letters  \(a\in\Adet\), all the stochastic letters \(A\in\Asto\) and, in addition, a deterministic letter \(A^\star\) associated with each \(A\in\Asto\). After setting \(\circ d\B_\ell(s) = ds\)
for all deterministic letters, we may define, via \eqref{eq:recurrenceintegrals}, Stratonovich iterated integrals \(J_w\) for each \(w\in\W^\star\), where \(\W^\star\) denotes the set of words for the alphabet \(\A^\star\). Note that this set of iterated integrals is different from that used to write the Stratonovich-Taylor expansion in \eqref{eq:taylor2}--\eqref{eq:taylornusol} because \(\W^\star\) is a larger set than \(\W\). With the \(J_w\), \(w\in \W^\star\), we construct the Chen series
\[
J^\star = \sum_{w\in\W^\star}J_w(t_0+h;t_0) w.
\]

The results in this section require the use of two mappings \(\theta\) and \(\rho\) that we introduce now. We
define \(\theta:\R\langle\langle \A^\star\rangle\rangle\rightarrow \R\langle\langle \eA\rangle\rangle\) as
follows. For letters, we set \(\theta (a) = a\) for \(a\in\Adet\), \(\theta (A)= A\) for \(A\in\Asto\) and
\(\theta (A^\star)= \bar A-(1/2)AA\) for \(A\in\Asto\). For words, we set \(\theta(\emptyset) = \emptyset\)
and \(\theta (\ell_1\dots\ell_n)= \theta\ell_1 \cdots \theta\ell_n\). We note that, for each \(w\in\W^\star\),
\(\theta (w)\) is a linear combination of words of weight \(\|w\|\). Finally, we set \(\theta (\sum_w S_w w)
=\sum_w S_w\theta w\). Clearly \(\theta\) is linear and in addition is an algebra morphism, i.e.\ maps the
concatenation \(S_1S_2\) into the concatenation \(\theta (S_1) \theta (S_2)\).

We next define a bilinear mapping \(\R\langle\langle \A^\star\rangle\rangle\times\R\langle \eA\rangle\rightarrow \R \) as in \eqref{eq:bilinear} and
define \(\rho:\R\langle \eA\rangle\rightarrow \R\langle \A^\star\rangle\) by demanding
\[
\big(\theta(S), p\big) = \big(S,\rho (p)\big)
\]
for each \(S\in \R\langle\langle \A^\star\rangle\rangle\) and each \(p\in \R\langle \eA\rangle\); thus \(
\rho\) is the linear map obtained from \(\theta\) by transposition with respect to  \((\cdot,\cdot)\). As an
example of the computation of \(\rho\), let us find \(\rho (AA)\). By definition, \(\theta (A^\star) = \bar
A-(1/2)AA\) and \(\theta (AA) = \theta (A)\theta (A)=AA\); for words \(w\) other than \(A^\star\) and \(AA\),
\((\theta(w), AA) = 0\) and therefore \(\rho (AA) = AA-(1/2) A^\star\). In general
\[\rho(w) = w+\sum_u \left( -\frac{1}{2}\right)^r u,\]
where the sum is extended to all words that may obtained by replacing  pairs of consecutive stochastic letters
\(AA\) by the corresponding \(\bar A\) and \(r\geq 1\) is the number of pairs replaced. For instance, \(\rho(aAAA) =
aAAA-(1/2)a\bar A A-(1/2)aA \bar A\) and \(\rho(AAAA) = AAAA -(1/2)\bar A AA-(1/2) A\bar A A-(1/2)AA\bar
A+(1/4) \bar A\bar A\) and \(\rho(AB) = AB\) if \(A\neq B\).

The maps \(\theta\) and \(\rho\) have been defined so that they encapsulate the relation between Ito and
 Stratonovich integrals, as shown in the next result, where the first formula expresses each Ito iterated integral
  as a linear combination of Stratonovich iterated integrals (cf. formula (8) in \cite{gaines}).

\begin{prop}\label{prop:IJ}
 For each \(w\in\overline W\),
\begin{equation}\label{eq:IJ}
I_w(t_0+h;t_0) = \Big(J^\star(t_0+h;t_0),\rho(w)\Big),
\end{equation}
and, therefore, for each \(p\in \R\langle \eA\rangle\),
\[
\Big(I(t_0+h;t_0),p\Big) = \Big(J^\star(t_0+h;t_0),\rho(p)\Big).
\]
As a consequence,
\(\theta\) maps the (Stratonovich) Chen series \(J^\star\)  into the (Ito) Chen series
\[
I = \sum_{w\in\overline\W} I_w(t_0+h;t_0) w.
\]
\end{prop}
\begin{proof}
The equality in \eqref{eq:IJ} clearly holds if \(w\) is empty or consists of a single stochastic letter.
Suppose that it holds for all words with weight  \(\leq N\), \(N\geq 1/2\), and consider a word of of weight \(N+1/2\), which we write in the form \(wk\ell\). Assume first that \(k=\ell=A\) for some \(A\in\Asto\).
By the recurrence relation between  iterated integrals, we have
\[
I_{wAA} = \int_{t_0}^{t_0+h} I_{wA}(s)\, d\B_A(s),
\]
and then, by the relation between Ito and Stratonovich stochastic integrals (see e.g.\cite[Section 3.2]{pav}, the induction hypothesis and \eqref{eq:recurrenceintegrals}, we may write
\begin{eqnarray*}
I_{wAA} &=& \int_{t_0}^{t_0+h} I_{wA}(s)\circ d\B_a(s)-\frac{1}{2} \int_{t_0}^{t_0+h} I_w(s)\, dt\\
&=& \int_{t_0}^{t_0+h} \Big(J^\star(s;t_0),\rho(wA)\Big)\circ d\B_A(s)-\frac{1}{2} \int_{t_0}^{t_0+h}
\Big(J^\star(s;t_0),\rho(w)\Big)\, ds\\
&=& \Big(J^\star,\rho(wA)A\Big)-\frac{1}{2}\Big(J^\star,\rho(w)\bar A\Big)\\
&=& \Big(J^\star,\rho(wAA)\Big).
\end{eqnarray*}
For other combinations of \(k\) and \(\ell\) one proceeds similarly, starting from
\[
I_{wk\ell} = \int_{t_0}^{t_0+h} I_{wk}(s)\, d\B_\ell(s) = \int_{t_0}^{t_0+h} I_{wk}(s)\circ d\B_\ell(s).
\]
\end{proof}

As a simple instance of \eqref{eq:IJ} we note that, from the relation \(\rho(AA)=AA-(1/2)A^\star\) found
above, we get \(I_{AA} = J_{AA}-(1/2)J_{A^\star}\), i.e.\ the well-known relation
\begin{eqnarray*}
&&\int_{t_0}^{t_0+h} \B_A(s)\,d\B_A(s)\\
 &&
\qquad = \int_{t_0}^{t_0+h} \B_A(s)\circ d\B_A(s)-\frac{1}{2}h= \frac{1}{2}\Big(\B_A(t_0+h)^2-\B_A(t_0)^2-h\Big).
\end{eqnarray*}

\subsection{The equivalence Ito--Stratonovich}

Proposition~\ref{prop:IJ} links the  Chen series \(J^\star\) and \(I\). We  investigate next the link between
the corresponding series of differential operators. Recall that, associated with each   \(a\in\Adet\) or each
\(A\in\Asto\), there is a first order (Lie) differential operator \eqref{eq:differentialoperator}. On the
other hand, letters \(\bar A\in\eA\) corresponding to \(A\in\Asto\) give rise to second order differential
operators \eqref{eq:differentialoperatorbis}. We now associate with each letter \(A^\star\), \(A\in\Asto\), the
first-order differential operator defined by
\[
D_{A^\star}\chi(x) = \chi^\prime(x)\Big(-\frac{1}{2}f_A^\prime(x)f_A(x)\Big).
\]
Thus \(D_{A^\star}\) is the Lie operator of the vector field \(-(1/2)f^\prime_A(x)f_A(x)\). With this
definition, a simple computation yields
\[
D_{A^\star} = D_{\bar A} -\frac{1}{2}D_{AA}
\]
i.e.\ \(D_{A^\star} = D_{\theta(A^\star)}\). Furthermore, for \(a\in\Adet\), \(\theta(a) = a\) and, for
\(A\in\Asto\), \(\theta(A)=A\) and therefore \(D_\ell = D_{\theta(\ell)}\) for each \(\ell\in\A^\star\). It
follows that \(D_{S} = D_{\theta(S)}\) for each series \(S\in \R\langle\langle \A^\star\rangle\rangle\). In
particular, from the last equality in Proposition~\ref{prop:IJ}, we conclude \(D_{J^\star} = D_{I}\), or, in
other words, the pullback operator \(D_I\) for the Ito equation \eqref{eq:ito} coincides with the pullback
operator \(D_{J^\star}\) of the Stratonovich equation
\begin{equation}\label{eq:stratonovich2}
dx =\sum_{a\in\Adet} f_a(x)\, dt-\frac{1}{2}\sum_{A\in\Asto}f^\prime_A(x)f_A(x)\,dt+\sum_{A\in\Asto} f_A(x)\circ d\B_A.
\end{equation}
In fact, as is well known, this Stratonovich equation and \eqref{eq:ito} have the same solutions.
 This is easily proved: \eqref{eq:intchiito} coincides with the result of writing formula \eqref{eq:intchi}
  for the  system \eqref{eq:stratonovich2}. Of course, if all the \(f_A\), \(A\in\Asto\) are constant (additive
  noise), \eqref{eq:stratonovich2} is the same as \eqref{eq:stratonovich}, i.e.\ \eqref{eq:ito} and
  \eqref{eq:stratonovich} share the same solutions, see e.g. \cite[Section 4.9]{kloeden}.

\section{Additional algebraic results}
In this section we briefly relate the preceding material to standard results on combinatorial (Hopf) algebras
and provide additional algebraic results. \emph{Hopf} algebras are important tools in the study of numerical
integrators and in other
 fields including e.g.\ renormalization theories; a very readable introduction that  requires little algebraic background
  is presented in \cite{brouder}. For instance many developments of Butcher's theory of Runge-Kutta methods
  may be phrased in the language of the Hopf algebra of trees and in fact Butcher anticipated many results
  on that algebra later
  rediscovered in different settings. Useful references are \cite{anderfocm,fm}.

The (associative, commutative) \emph{shuffle} algebra \(\Hs(\A)\) of the alphabet \(\A\) is defined as
follows. As a vector space \(\Hs(\A)\) coincides with \(\R\langle \A\rangle\). However the product in
\(\Hs(\A)\) is given by shuffling words rather than by concatenating them. The algebra \(\Hs(\A)\) is graded
by the weight \(\|\cdot\|\). In addition we may consider in \(\Hs(\A)\) a \emph{coproduct} by decomposing
(deconcatenating) each word \(w\in\W\) \(\ell_1\dots \ell_n\) as
 \[
 \emptyset \otimes \ell_1\dots \ell_n+\ell_1\otimes \ell_2\dots \ell_n+\ell_1\dots \ell_n\otimes \emptyset.
 \]
 This coproduct is compatible with the shuffle product because, as explained in the proof of
 Lemma~\ref{lem:shuffle}, the operations of shuffling and deconcatenation commute. Therefore \(\Hs(\A)\) is a
 Hopf algebra graded by the weight.

 The dual vector space of \(\Hs(\A)\) may be identified with the vector space of formal series
 \(\R\langle\langle \A\rangle\rangle\) via the bilinear form \eqref{eq:bilinear}. In other words, the linear
 form on \(\Hs(\A)\) that as \(w\) ranges in \(\W\) associates with \(w\)  the real number \(S_w\) is
 identified
 with \(\sum S_w w\). With this identification, the concatenation product  of series \(S\in\R\langle\langle
 A\rangle\rangle\), or equivalently the product \eqref{eq:convolution} for the coefficents, coincides with the
 convolution product in the dual of the Hopf algebra. Series \(S\) with \(S_\emptyset = 1\) that satisfy the
 shuffle relations are then the linear forms on \(\Hs(\A)\) that are multiplication morphisms (i.e.\ preserve
 multiplication). The set of those linear forms
 forms is well known to be a group for the convolution product;  this group is called the \emph{shuffle group}
 and denoted \(\Gs(\A)\).
 Therefore  Lemma~ \ref{lem:shuffle} is just the statement that the convolution product of two elements in
 \(\Gs(\A)\) lies in \(\Gs(\A)\).

 The \emph{quasishuffle Hopf algebra} \(\Hq(\bar A)\) is constructed similarly. One endows the vector space
 \(\R\langle \eA\rangle\) with the quasishuffle product and the deconcatenation coproduct. The series
 \(S\in\R\langle\langle \eA\rangle\rangle\) with \(S_\emptyset = 1\) that satisfy the quasishuffle relations
 may then be viewed as forming the \emph{quasishuffle group} \(\Gq\) of linear forms on \(\Hq(\bar A)\) that are
 multiplication morphisms.

Theorem 2.5 in \cite{hoffman} shows that the mapping \(\rho\) is an isomorphism of \(\Hq(\eA)\) onto
\(\Hs(\A^\star)\). In particular it maps the quasishuffle product into the shuffle product:
\[
\rho(u\bowtie v) = \rho(u)\sh \rho(v), \qquad \forall u,v\in\overline W.
\]
This observation and the material in Section~\ref{sec:is} make clear that the quasishuffle/Ito results in
Propositions~\ref{prop:quasishuffle}--\ref{prop:Itildemultiplicative} may be derived from the corresponding
shuffle/Stratonov\-ich results by transforming \(\sh\) into \(\bowtie\) with the help of the inverse
isomorphism \(\rho^{-1}\).
\section*{Acknowledgements}
 J.M.S. has been supported by project MTM2016-77660-P(AEI/FEDER,
UE) funded by MINECO (Spain).

\end{document}